%% file: ArticlePotts.tex
\documentclass[11pt,a4paper]{article}
\usepackage[english]{babel}       \usepackage[latin1]{inputenc}
\usepackage[T1]{fontenc}          \usepackage{a4}
\usepackage[cyr]{aeguill}         \usepackage{epsfig}
\usepackage{amsmath,amsthm,amsfonts,amssymb,mathrsfs}
\usepackage{pgf,tikz}               \usetikzlibrary{arrows}
\usepackage{dsfont}                \usepackage{float}
\usepackage{fancyhdr}              \usepackage{varioref}
\usepackage{color}     \usepackage{url}         \usepackage{calc}
\usepackage{wrapfig}               \usepackage{graphicx}
\usepackage{enumitem,hyperref}     \usepackage{pifont} 
\newtheorem{theorem}{Theorem} 
\newtheorem{definition}{Definition}[section]
\newtheorem{remark}{Remark}[section]
\newtheorem{proposition}{Proposition}[section]
\newtheorem{lemma}{Lemma}[section]

\numberwithin{equation}{section}

\usepackage{authblk}

\makeatletter
\def\namedlabel#1#2{\begingroup
    #2%
    \def\@currentlabel{#2}%
    \phantomsection\label{#1}\endgroup
}
\makeatother
\input{macros}

\begin{document}
\title{Phase transition for the non-symmetric Continuum Potts model}

\author{Pierre Houdebert}
\affil{Department of Mathematics, University of Potsdam}
\affil{
 \texttt{pierre.houdebert@gmail.com}}
 \date{}
\maketitle

\begin{abstract}
{We prove a phase transition for the non-symmetric continuum Potts model with background interaction, by generalizing the methods introduced in the symmetric case by Georgii and Häggström \cite{georgii_haggstrom}.
The proof relies on a Fortuin-Kasteleyn representation, percolation and stochastic domination arguments.  
} 
  \bigskip

  \bigskip

\end{abstract}
\section{Introduction} \label{section_introduction}
In Gibbs point processes theory one of the main question of interest is the study of phase transition.
Indeed Gibbs point processes are defined through a family of equations, the Dobrushin-Lanford-Ruelle equations, and it is a natural question to ask whether there exists only one or several solutions to these equations.
Although phase transition is conjectured for most continuum models, it has been rigorously proved only in a few cases.
The first such result was obtain by Ruelle \cite{ruelle_1971} for the symmetric Widom-Rowlinson model, which is a two type particles system with an hard-core repulsion between particles of different types, using a continuum version of the Peierls argument.
This technique was latter generalized to the soft-core Widom-Rowlinson interaction in \cite{lebowitz_lieb_1972}.

In the 1990's Chayes, Chayes \& Koteck\'y \cite{chayes_kotecky} 
and Georgii \& Häggström \cite{georgii_haggstrom} generalized for continuum models the idea of the Fortuin-Kasteleyn representation \cite{fortuin_kasteleyn_1972} introduced for the lattice Ising and Potts models, and proved phase transition results respectively for the symmetric Widom-Rowlinson model and for the continuum Potts model with background interaction.
This idea was then used in a variety of articles, for instance to prove phase transition for the symmetric Widom-Rowlinson model with unbounded radii
 \cite{Dereudre_Houdebert_2019_JSP_PhaseTransitionWR,
Houdebert_2017_percolation_CRCM}.
The idea of the Fortuin-Kasteleyn representation is generalized to the non-symmetric case in the present article.

For the non-symmetric case where each type of particles have different intensities, even fewer results are proved.
A few results are proved for the Widom-Rowlinson model using the Pirogov-Sinai technique, see for instance \cite{bricmont_kuroda_lebowitz_1984,
mazel_suhov_stuhl}.
Recently a sharp phase transition result for the Widom Rowlinson model was obtained in \cite{Dereudre_Houdebert_2019_sharp_phase_transition_WR}, giving an almost complete picture of the phase diagram.

In this article we are interested in the Continuum Potts model with background interaction, as introduced by Georgii and Häggström \cite{georgii_haggstrom}.
We prove that for any initial proportion of particles 
$\Proportion=(\alpha_1,\dots,\alpha_\Color)$ and for the activity parameter $\Intensity$ large enough, there is at least as many distinct Potts measures as there are 
$\alpha_i, i=1\dots \Color$ which are maximal in $\Proportion$.
This result and its proof is a generalization of the proof of the symmetric case done by Georgii and Häggström \cite{georgii_haggstrom}.

The proof relies on a Fortuin-Kasteleyn representation which expresses the colouring correlation as the connectivity in the so-called generalized Continuum Random Cluster model.
This is done using stochastic domination tools.
Therefore by proving a percolation-type bound for this process, one can construct different Potts measures obtained by having different boundary conditions.
Such an idea was used  in \cite{biskup_borgs_chayes_kotecky} for the lattice nearest neighbour Potts model, and we are generalizing it for the continuum setting, to obtain a phase transition result with the \emph{exact same} assumptions as in \cite{georgii_haggstrom}.

The article is organized as follow: in Section \ref{section_preliminaries} we introduced the model and the tools needed later on. 
In Section \ref{section_results} we give the assumptions and state the theorems. 
In Section \ref{section_fk_representation} is introduced the Fortuin-Kasteleyen representation and we state and prove the percolation bound for the generalized Continuum Random Cluster model.
In Section \ref{section_proofs_theo} we prove the main theorems, and finally in the appendix Section \ref{section_appendix} we give the proofs of classical and technical lemmas.
\section{Preliminaries} \label{section_preliminaries}
\subsection{Space}
Through the paper the dimension $\Dim\geq2$ and the number of colours $\Color \geq 2$ are fixed integer numbers.
We are considering the space $\ConfSpace$ (respectively $\ConfSpace_\Lambda$) of locally finite configurations $\Conf$ in $\R^\Dim$ (respectively $\Lambda$).
We will often consider configurations marked by a colour.
For those we are using the notation 
\begin{align*}
\sigma\colon\begin{cases}
\Conf \longrightarrow \{1,\dots,\Color \}\\
x\longmapsto \sigma_x
\end{cases}
\end{align*}
and we write respectively
$\tConfSpace=\{\tConf=(\Conf,\sigma) \}$ 
The configurations spaces $\ConfSpace$ and $\tConfSpace$ are embedded with the usual sigma-algebras generated by the counting variables.

For $\Lambda \subseteq \R^\Dim$, we write $\Conf_{\Lambda}$ as a shorthand for $\Conf \cap \Lambda$.
This notation naturally extends to 
$\tConf_{\Lambda}$.
We write $N_\Lambda(\Conf)$ (respectively 
$N_\Lambda(\tConf)$) 
for the cardinality of the respected configuration inside $\Lambda$.
We write $|\Lambda|$ for the Lebesgue measure of $\Lambda \subseteq \R^\Dim$, and $|j|$ for the sup norm of $j \in \Z^\Dim$.
We write $\Conf' \Conf$ (respectively $\tConf' \tConf$) has a shorthand for $\Conf' \cup \Conf$ (respectively $\tConf' \cup \tConf$).

Let $\tau_x$  be the translation of vector $x \in \R^\Dim$.
This means that $\tau_x(\Conf)=\{y-x, \ y \in \Conf  \}$.
We denote by $\tProbaSetInvariant$ (respectively $\tProbaSetInvariantDelta$)  the set of probability measures on $\tConfSpace$ 
which are invariant under all translations of $\R^\Dim$ (respectively all translation in $\delta\Z^\Dim$).
For $\delta>0$, we write $\Delta_j=j \oplus ]-\delta/2, \delta/2]^\Dim$ 
with $j \in \delta \Z^\Dim$.
\subsection{Poisson point processes}\label{section_PPP}
Let $\Poisson^{\Intensity}$ be the distribution on $\ConfSpace$ of the homogeneous Poisson point process with intensity $\Intensity >0$. 
Recall that it means 
\begin{itemize}
\item for every bounded Borel set $\Lambda$, the distribution of the number of points in $\Lambda$ under $\Poisson^\Intensity$ is a Poisson distribution of mean $\Intensity |\Lambda|$;
\item given the number of points in a bounded $\Lambda$, the points are independent and uniformly distributed in $\Lambda$.
\end{itemize}
We refer to \cite{daley_vere_jones} for details on Poisson point processes.
We write
$\tPoisson^{\Intensity,\Proportion}$ for the distribution on $\tConfSpace$ of the Poisson point process of intensity $\Intensity$ with independent colour marks distributed according to a probability measure $\Proportion=(\alpha_1,\dots,\alpha_q)$ on $\{1,\dots,q\}$.
We have $\tPoisson^{\Intensity,\Proportion} \in \tProbaSetInvariant$.
We are assuming, without loss of generality, that $\Proportion$ has non-zero marginals, i.e. $\alpha_i >0$ for all $i$.

For $\Lambda \subseteq \R^d$, we denote by $\Poisson^{\Intensity}_\Lambda$ 
(respectively 
$\tPoisson^{\Intensity,\Proportion}_\Lambda$)
the restriction of  $\Poisson^{\Intensity}$ 
(respectively 
$\tPoisson^{\Intensity,\Proportion}$)
on $\Lambda$.
\subsection{Continuum Potts model with background interaction}
For $\Lambda \subseteq \R^\Dim$ bounded, we define the $\Lambda$-Hamiltonian $\Hamiltonian_{\Lambda}$ such that, for $\tConf \in \tConfSpace$,
\begin{align*}
\Hamiltonian_{\Lambda}(\tConf)
:=
\displaystyle 
\sum_{\substack{\{x,y\} \subseteq \Conf 
\\ \sigma_x \not = \sigma_y 
\\ \{x,y\}\cap \Lambda \not = \emptyset } }
\phi(x-y)
+
\displaystyle 
\sum_{\substack{\{x,y\} \subseteq \Conf 
\\ \{x,y\}\cap \Lambda \not = \emptyset } }
\psi(x-y)
:=
\Hamiltonian^{\phi}_{\Lambda}(\tConf) 
+ \Hamiltonian^{\psi}_{\Lambda}(\Conf),
\end{align*}
where $\phi,\psi: \R^\Dim \to ]-\infty,+\infty]$ are even measurable functions.
The first potential $\phi$ describes a repulsion between points of different colours.
The second $\psi$ is a type-independent pair potential.
The most classical Potts model is the Widom-Rowlinson model \cite{widom_rowlinson}, for which $\psi=0$ and $\phi(x)=+\infty \1_{|x| \text{ small}}$.

\begin{definition}
For a boundary condition $\tConf$, we define the Potts specification on a bounded $\Lambda\subseteq \R^\Dim$ as
\begin{align*}
\tSpecification^{\Intensity,\Proportion}_{\Lambda,\tConf}(d\tConf'_{\Lambda})
=\frac{\exp (-
\Hamiltonian_{\Lambda}(\tConf'_{\Lambda} \tConf_{\Lambda^c}) )}
{\PartitionFunction_{\Lambda}(\tConf)}
\tPoisson^{\Intensity,\Proportion}_{\Lambda}(d\Conf'_{\Lambda}),
\end{align*}
where $\PartitionFunction_{\Lambda}(\tConf)
= \int_{\tConfSpace}
\exp (-
\Hamiltonian_{\Lambda}(\tConf'_{\Lambda} \tConf_{\Lambda^c}) )
\tPoisson^{\Intensity,\Proportion}_{\Lambda}(d\Conf'_{\Lambda})
$ is the partition function.
\end{definition}
At this point nothing ensures the well-definedness of the Potts specification.
Conditions ensuring the well-definedness of the Potts specification will be introduced later.
\begin{definition}
A probability measure $P$ on $\tConfSpace$ is a Potts measure of potentials $\phi,\psi$, of activity $\Intensity$ and of colour proportion $\Proportion$, written $P \in \GibbsPotts(\Intensity, \Proportion)$,
if for every bounded $\Lambda \subseteq \R^\Dim$ and every bounded measurable function $f$, we have 
$0< Z^{\Intensity,\Proportion}_{\Lambda}(\tConf)<\infty$ for $P(d\tConf)$ almost every configuration, and
\begin{align}
\label{eq_dlr_potts_potentiel}
\int_{\tConfSpace} f dP
=
\int_{\tConfSpace} \int_{\tConfSpace_\Lambda} 
f(\tConf'_{\Lambda}  \tConf_{\Lambda^c})
\tSpecification^{\Intensity,\Proportion}_{\Lambda,\tConf}(d\tConf'_{\Lambda})
P(d\tConf).
\end{align}
We write $\GibbsPotts_{\theta}(\Intensity, \Proportion)$ as a shorthand for the set of Potts measures which are invariant under all translations of $\R^\Dim$, i.e.
$\GibbsPotts_{\theta}(\Intensity, \Proportion)
:=
\GibbsPotts(\Intensity, \Proportion) \cap \tProbaSetInvariant.$
\end{definition}
The equations \eqref{eq_dlr_potts_potentiel}, for every $\Lambda$, are called DLR equations, named after Dobrushin, Lanford and Ruelle.
They prescribe the conditional probability kernels of a Potts measure.
\begin{remark}
In \cite{georgii_haggstrom}, they define the Potts measures on the set of \emph{tempered configurations}.
In our proof the measure built will be supported on the tempered configurations.
However it is not necessary to impose Potts measures to be supported on the set of tempered configurations.
The existence of a Potts measure which is not supported on the set of tempered configurations remains an open problem.
\end{remark}
\section{Results}  \label{section_results}
In the theory of infinite volume Gibbs probability measures, the Gibbs measures are defined through a family of equations, 
the \emph{DLR equations} \eqref{eq_dlr_potts_potentiel}.
With such definition the questions of existence and uniqueness/non-uniqueness of the defined objects are natural and interesting questions studied by the statistical mechanics community for a variety of interactions.
In the following we are stating an existence result and a phase transition (meaning the non-uniqueness of the Potts measures) result.

We are considering the following assumptions on $\phi, \psi$:
there exist $u>0$ and 
$0 \leq r_1 \leq r_2 <r_3 \leq r_4 < \infty$ such that
\newline
\begin{itemize}
\item[(A1)] (strict repulsion of $\phi$) 
$\phi \geq 0$ and $\phi(x)\geq u$ when $|x|\leq r_3$;
\item[(A2)] (finite range of $\phi$) 
$\phi(x)=0$ when $ |x| \geq r_4$;
\item[(A3)] (strong stability and regularity of $\psi$)
either $\psi \geq 0$, or $\psi$ is superstable and lower regular in the sense of Ruelle, meaning that
\begin{itemize}
\item (superstability) there exist constants $a,b>0$ such that for every finite configuration $\Conf$,
\begin{align*}
H^{\psi}(\Conf)
:=
H^{\psi}_{\R^\Dim}(\Conf)
\geq
\sum_{j \in \Z^\Dim}
 \left( \ a N_{\Delta_j}(\Conf)^2 -bN_{\Delta_j}(\Conf) \ \right);
\end{align*}
\item (lower regularity)
there exist positive numbers $\psi_n, \ n \in\N$, such that 
$\sum_{n \in \N}n^{\Dim -1}\psi_n < \infty$ 
and such that for every configuration $\Conf$,
\begin{align*}
\sum_{x \in \Conf_{\Delta_k}}
\sum_{y \in \Conf_{\Delta_j}}
\psi(x-y)
\geq
-\psi_{\delta^{-1}|j-k|}
N_{\Delta_k}(\Conf) N_{\Delta_j}(\Conf);
\end{align*}
\end{itemize}
\item[(A4)] (short range of repulsion for $\psi$) 
$\psi(x)\leq 0$ when $|x| > r_2$, and the positive
part $\psi^+$ of $\psi$ satisfies
\begin{align*}
\int_{|x|\geq r_1}\psi^+(x)\ dx <\infty;
\end{align*}
\item[(A5)] (scale relations) 
$r_2 < r_3 /2\sqrt{d + 3}$, and $r_1$ is sufficiently small (but independent of $\Proportion$, see \eqref{eq_assumption_A5}).
\end{itemize}
A classical model satisfying these assumptions is the \emph{Widom-Rowlinson model} \cite{widom_rowlinson}, for which $\Psi=0$ and $\Phi(x)=+ \infty \1_{|x| \leq 1}$.

Those assumptions are \emph{exactly} the same as the one considered by Georgii and Häggström in \cite{georgii_haggstrom}.
In their paper they are considering the symmetric case (i.e. $\alpha_i=1/ \Color$ for all $i$) and are proving a phase transition result.
Our result generalized their approach to prove phase transition for the non-symmetric case.

Our first theorem states the existence of at least one translation invariant Potts measure.
\begin{theorem}\label{theo_existence_Potts}
Assume that assumptions (A1) to (A3) are satisfied.
Then for every $\Intensity$ and every $\Proportion$, there exists at least one Potts measure $P \in \GibbsPotts_{\theta}(\Intensity, \Proportion)$, which is ergodic with respect to the translation group $(\tau_x)_{x \in \R^\Dim}$.
\end{theorem}
The second theorem states a phase transition for large enough $\Intensity$.
To state it, let us first define $\ProportionMax$ as the number 
of colours that have maximal proportion in $\Proportion$, i.e.
\begin{align*}
\ProportionMax
=
card \{i=1\dots \Color\ | \ 
\alpha_i \geq \alpha_{i'} \text{ for all } i' \not = i \}.
\end{align*}
This quantity is between $1$ and $\Color$.
For simplicity we are assuming that the colour $1$ is one of the colours with maximal proportion, i.e $\alpha_1 \geq \alpha_i$ for all $i=1\dots \Color$.

In the symmetric case when $\alpha_i=1 / \Color$ for all $i$, which means that 
$\Color=\ProportionMax$, Georgii and H\"aggstr\"om \cite{georgii_haggstrom} proved for large enough $\Intensity$ the existence of at least $q$ ergodic Potts measures.
The following theorem generalizes their result to the non-symmetric case.
\begin{theorem} \label{theo_phase_transition_Potts}
Assume that assumptions (A1) to (A5) are satisfied.
Then for $\Intensity$ large enough, depending on $\Color$, $u$ and $r_1$ to $r_4$, but independent of $\Proportion$, there exists at least $\ProportionMax$ Potts measures for $\phi,\psi,\Intensity,\Proportion$ which are ergodic with respect to the translation group $(\tau_x)_{x \in \R^\Dim}$.
\end{theorem}
This theorem does not give any indication in the case when $\ProportionMax=1$. 
We are conjecturing that in this case there is no phase transition, at least when $\Color$ is not too large.
This conjecture is motivated by similar result proved for the (lattice) nearest neighbour Ising model, see \cite{friedli_velenik_2017} for instance.
Recently this conjecture was partially solved in the specific case of Widom-Rowlinson model ($\Color=2$, $\psi=0$ and $\phi(x)=\infty \1_{|x| small}$) in 
\cite{Dereudre_Houdebert_2019_sharp_phase_transition_WR}: 
they proved that for large activity $\Intensity$, phase transition is only possible in the symmetric case.

The idea of the proof of Theorem \ref{theo_phase_transition_Potts} is the same as in \cite{georgii_haggstrom}: 
a Fortuin-Kasteleyn representation and a percolation bound uniform in the volume which pass through the limit.
The novelty is the introduction of the \emph{generalized Continuum Random Cluster model}.
This is a random connection model with an interaction depending on the number of connected components and their sizes.
This model allows the construction of a Fortuin-Kasteleyn representation, even in the non-symmetric case.
Such an idea was already used for the (lattice) nearest-neighbour Potts model in \cite{biskup_borgs_chayes_kotecky}, from which the terminology \emph{generalized Random Cluster Model} was taken from.

One other question is the uniqueness of the Potts measure.
It is in general conjectured that uniqueness occurs when the activity $\Intensity$ is small enough.
When $\Psi \geq 0$, the Potts specification is stochastically dominated by a Poisson point process.
This stochastic domination leads, using the technique of \emph{disagreement percolation}, to the uniqueness of the Potts measure when $\Intensity$ is small.
Indeed in the case $\Psi=0$, the Potts model falls into the general assumptions of the result proved in \cite{Hofer-temmel_Houdebert_2018}.
In the general case where $\Psi$ can be negative, disagreement percolation does not apply anymore.
In \cite{georgii_haggstrom} the authors claim that an extension of the
\emph{Dobrushin uniqueness criterion} could be used in order to prove uniqueness.
However, to the best of our knowledge, no such result exists in the literature.
One alternative could be to consider using \emph{cluster expansion},  which is better suited for potential with negative part.

The rest of the article is divided as follows: 
in Section \ref{section_fk_representation} we are introducing the Fortuin-Kasteleyn representation and proving a percolation bound for the generalized Continuum Random Cluster model.
In Section \ref{section_proofs_theo} we are proving Theorem \ref{theo_existence_Potts} and Theorem \ref{theo_phase_transition_Potts}.
Finally in the appendix in Section \ref{section_appendix} we are proving some technical lemmas used during the previous sections.
\section{Fortuin-Kasteleyn representation and percolation bound for the generalized Continuum Random Cluster model}\label{section_fk_representation}
At the core of the proof of Theorem \ref{theo_phase_transition_Potts} lies a representation of the Potts model called \emph{Fortuin-Kasteleyn} representation which provides the mean proportion of each colour in the Potts model, expressed as connectivity probabilities in a percolation model.
This representation was introduced first by Edwards and Sokal and then used to prove phase transition results in many models, including the symmetric Widom-Rowlinson model \cite{chayes_kotecky} and more generally continuum Potts models \cite{georgii_haggstrom}.
It needs to study connectivity in the so-called Continuum Random Cluster model, which is a Gibbs model with an interaction depending only on the number of connected components.
This model was first introduced in \cite{klein_1982} and then used in \cite{georgii_haggstrom} and \cite{chayes_kotecky} to prove phase transition, by providing an uniform bound (with respect to the finite volume box $\Lambda$) of the percolative probability that the boundary is connected to the origin.
The Continuum Random Cluster model was also studied on its own in \cite{dereudre_houdebert,Houdebert_2017_percolation_CRCM}.

In our approach we are generalizing this method to the non-symmetric case and introducing the \emph{generalized Continuum Random Cluster model}.
This is a continuum version of the generalized Random Cluster model,  used in \cite{biskup_borgs_chayes_kotecky} to prove phase transition for the non-symmetric lattice nearest neighbour Potts model.
\subsection{FK representation}
We consider the set  $\EdgeSet$ of locally finite families of edges of the form 
$E=\cup_{i \in I}{x_i,y_i}$ with $x_i \not = y_i$ are in $\R^\Dim$.
This set is endowed with the classical $\sigma$-algebra generated by the counting variables.
We write $\EdgeSet_{\Conf} \subseteq \EdgeSet$ for the families of edges between points of $\Conf$.

From now on we are fixing $\Lambda \subseteq \R^\Dim$ bounded.
We are defining a point process 
$
\mathbb{P}^{\Intensity,\Proportion}_{\Lambda
}$ on 
$\tConfSpace\times \EdgeSet$ in the following way:
\begin{itemize}
\item The distribution of points is given by the potential $\psi$:
\begin{align*}
P_{\Lambda}^{\Intensity,\psi}(d\Conf)
=
\frac{\exp \left( - H^{\psi}_{\Lambda}(\Conf) \right)}
{Z^{\psi}_{\Lambda}(
\emptyset)}
 \Poisson_{\Lambda}^{\Intensity}(d \Conf),
\end{align*}
where $Z^{\psi}_{\Lambda}(
\emptyset)$ is the corresponding partition function, which is well defined thanks to the stability of the potential $\psi$.
The "$\emptyset$" is there to emphasize that this point process measure is free of boundary condition.
Notice here that $P_{\Lambda}^{\Intensity,\psi}$ 
is a point process measure on $\ConfSpace_\Lambda$.
\item when the locations are known, the colours are independent  random variables of law 
$\Proportion$ with deterministic colour for the points too close to the boundary of $\Lambda$, i.e
\begin{align}\label{eq_coloration_ES}
\lambda^{\Proportion, 1}_{\Conf,\Lambda}(\tConf)
=
\frac{1}{Z^{\Proportion}_{\Lambda, \Conf}(1)}
\left(
\prod_{\substack{
x \in \Conf_{\Lambda}}}
 \alpha_{\sigma_x}
\right)
\1_{A^1_{r_4}}(\tConf)
,
\end{align}
where 
$ A^1_{r_4}
=
\{ \tConf | \sigma_x=1,\forall x \ s.t. \ dist(x,\Lambda^c)\leq r_4 \}$ and
$Z^{\Proportion}_{\Lambda, \Conf}(1)$ is the corresponding normalizing constant.
The "1" in the partition function is there to emphasis on the fact that the points close to the boundary of $\Lambda$ are coloured deterministically.
\item Finally the edge drawing mechanism between points of $\Conf$ is the probability 
$
\mu^{\phi}_{\Conf}$ on $\EdgeSet_\Conf$ such that 
\begin{align}\label{eq_edges_ES}
\mu^{\phi}_{\Conf}(E)
= 
\prod_{\substack{
\{x,y\}\subseteq \Conf
\\ \{x,y\}\in E }}
\left( 1-e^{-\phi(x-y)} \right) 
\prod_{\substack{
\{x,y\}\subseteq \Conf
\\ \{x,y\}\not \in E }}
e^{-\phi(x-y)}. 
\end{align}
\end{itemize}

The probability measure $\FK$ is then defined as the product measure
\begin{align*}
\FK(d\tConf,dE)
= 
\mu^{\phi}_{\Conf}(E) 
\lambda^{\Proportion, 1}_{\Conf,\Lambda}(\tConf)
P_{\Lambda}^{\Intensity,\psi}(d\Conf).
\end{align*}
Then we consider the event $\A$ on $\tConf \times \EdgeSet$ of \emph{authorized} configurations where every connected points have the same colour.
This event has positive $\FK$-probability, as the configuration empty 
of points in $\Lambda$ is authorized.
We can then consider the probability measure 
$\FKA:= 
\FK (.|\A)$.
\begin{remark}
In \cite{georgii_haggstrom}, they constructed the measure with a periodic boundary condition of points of type 1.
We could have made the same but we do believe that our construction, forcing points close to the boundary to be of colour 1, is easier to understand and closer to the constructions made for the Ising model for instance.

The indicator in \eqref{eq_coloration_ES} should be understood as if all points close to the boundary of $\Lambda$ are connected to a imaginary point "at infinity" which is of colour 1.
The connected component of points connected to this imaginary point "at infinity" will be called the \emph{infinite cluster} and written $\InfiniteCluster$.
A formal definition will be given later on.
Furthermore, in the definition on the event $A^1_{r_4}$, one does not need to take the same radius as in the condition (A2). 
Every choice of finite radius would work as well.
\end{remark}
\begin{proposition} \label{propo_projection_potts}
The projection of the measure 
$\FKA$ on $\tConfSpace$ is 
\begin{align*}
\tSpecification^{\Intensity,\Proportion}_{\Lambda,1}
(d \tConf_{\Lambda})
:=
\1_{A_{r_4}^1}(\tConf)
\frac{\exp (-
\Hamiltonian_{\Lambda}(\tConf_{\Lambda}) )}
{\PartitionFunction_{\Lambda}(1)}
\tPoisson^{\Intensity,\Proportion}_{\Lambda}(d\tConf_{\Lambda}).
\end{align*}
\end{proposition}
The measure $\tSpecification^{\Intensity,\Proportion}_{\Lambda,1}$ has to be understood the following way: 
one can imagine that at the boundary of $\Lambda$ there are a continuum of boundary points of colour 1, forcing points too close to $\Lambda$ to be of colour 1.
But this continuum of points do not give an interaction coming from $\psi$, only a colour exclusion.
\begin{proof}
Let $f$ be a measurable bounded function on $\tConfSpace$.
\begin{align*}
\int_{\tConfSpace \times \EdgeSet} 
f(\tConf)& \
\FKA(d\tConf, dE)
\\ &=
\FK(\A)^{-1}
\int_{\ConfSpace} \int_{\tConfSpace} f(\tConf)
\left[\int_{\EdgeSet}\1_{\A}(\tConf,E) \mu^{\phi}_{\Conf}(E) \right]
\lambda^{\Proportion, 1}_{\Conf,\Lambda}(\tConf)
P_{\Lambda}^{\Intensity,\psi}(d\Conf),
\end{align*}
and thanks to simple computation we have
\begin{align*}
\int_{\EdgeSet}  \1_{\A}(\tConf,E) & \mu^{\phi}_{\Conf}(dE)
 \\ &=
\sum_{\substack{
E \in \EdgeSet_{\Conf},  \\ \sigma_x=\sigma_y  \\ \forall \{x,y\}\in E }}
\prod_{\substack{
\{x,y\}\subseteq \Conf
\\ \{x,y\}\in E }}
\left( 1-e^{-\phi(x-y)} \right) 
\prod_{\substack{
\{x,y\}\subseteq \Conf
\\ \{x,y\}\not \in E }}
e^{-\phi(x-y)}
=
e^{-\Hamiltonian^{\phi}_{\Lambda}(\tConf)}
\end{align*}
 and therefore
\begin{align*}
\int_{\tConfSpace \times \EdgeSet} 
f(\tConf) 
\FKA(d\tConf, dE) 
& =
\Constant
 \int_{\tConfSpace} f(\tConf_{\Lambda})
e^{-\Hamiltonian_{\Lambda}(\tConf_{\Lambda})}
\1_{A_{r_4}^1}(\tConf)
\tPoisson_{\Lambda}^{\Intensity,\Proportion}(d\tConf_{\Lambda})
\\&=
\int_{\tConfSpace} f(\tConf_{\Lambda})
\tSpecification^{\Intensity,\Proportion}_{\Lambda,1}(d\tConf'_{\Lambda}).
\end{align*}
\end{proof}
We are considering now the projection of 
$\FKA$
 on $\ConfSpace \times \EdgeSet$.
We say that two points $x,y \in\Conf$ are connected in ($\Conf,E$) if there is a path $x_1,\dots,x_n$ with $x_1=x$, $x_n=y$ and such that $\{x_i,x_{i+1}\}\in E$ for all $i$.
A point $x$ such that $dist(x,\Lambda^c)\leq r_4$ is said to be linked (with an imaginary edge) to an imaginary point "at infinity".
We are then considering the connected components with respect to this connectivity rule, with the particularity that
all points $x$ connected to  infinity, i.e. such that $x$ is connected to a point $y$ linked to infinity, are said to be in the infinite connected component $\InfiniteCluster$.

The number of connected components is $\FK-a.s.$ finite, with at most one component connected "at infinity" $\InfiniteCluster$.
But let us emphasize that the cardinality $|\InfiniteCluster|$ of $\InfiniteCluster$ is 
$\FK-a.s.$ finite since $\InfiniteCluster$ is a configuration contained in $\Lambda$.

Let us now consider the measure 
$\RCM$ on 
$\ConfSpace \times \EdgeSet$, called \emph{generalized Continuum Random Cluster model} on $\Lambda$ with wired boundary condition
, defined as
\begin{align*}
\RCM(d\Conf,dE)
=\frac{\alpha_1^{|\InfiniteCluster
|}}{Z^{g}_{\Lambda}(\Proportion)}
\prod_{ 
\substack{
\InfiniteCluster \not = C \subseteq \Conf
\\ \text{cluster of } (\Conf,E) }} 
\left( \sum_{i=1}^{\Color} \alpha_i^{|C|} \right)
\mu^{\phi}_{\Conf}(E) 
P_{\Lambda}^{\Intensity,\psi}(d\Conf),
\end{align*}
with $Z^{g}_{\Lambda}(\Proportion)$ being the associated partition function.
\begin{proposition} \label{proposition_projection_gcrcm}
The projection of $\FKA$ 
on $\Conf \times \mathcal{E}$ is 
$\RCM$.
\end{proposition}
\begin{proof}
Let $f$ be a measurable bounded function on $\ConfSpace \times \EdgeSet$.
\begin{align*}
\int_{\tConfSpace \times \EdgeSet} 
f \ d \FKA
&=
\FK(\A)^{-1}
\int_{\ConfSpace} \int_{\EdgeSet}  
f(\Conf,E)
\left[
\int_{\tConfSpace} \1_{\A}(\tConf,E) 
\lambda^{\Proportion, 1}_{\Conf,\Lambda}(\tConf)  \right]
\mu^{\phi}_{\Conf}(E) 
P_{\Lambda}^{\Intensity,\psi}(d\Conf),
\end{align*}
but thanks to the product structure of the measure 
$\lambda^{\Proportion, 1}_{\Conf,\Lambda}$, and denoting by 
$C_i \subseteq \Conf$ the finite (i.e. not connected "at infinity") connected components of $(\Conf,E)$, we have
\begin{align*}
\int_{\tConfSpace} \1_{\A}(\tConf,E)
 \lambda^{\Proportion, 1}_{\Conf,\Lambda}(\tConf) 
& =
\frac{\alpha_1^{|\InfiniteCluster|}}{Z^{\Proportion}_{\Lambda, \Conf}(1)}
\sum_{\tConf_{C_1}|\Conf_{C_1}} 
\dots \sum_{\tConf_{C_n}|\Conf_{C_n}}
\prod_{i=1}^{n} \1_\A(\tConf_{C_i},E)
\prod_{x \in \Conf_{C_i}}\alpha_{\sigma_{x}}
\\ &= 
\frac{\alpha_1^{|\InfiniteCluster|}}{Z^{\Proportion}_{\Lambda, \Conf}(1)}
\prod_{ 
\substack{
\InfiniteCluster \not = C \subseteq \Conf
\\ \text{cluster of } (\Conf,E) }} 
\left( \sum_{i=1}^{\Color}\alpha_i^{|C|} \right),
\end{align*}
which implies the wanted result.
Here the sum 
$\underset{\tConf|\Conf}{\sum}$
 is over all coloured configurations
$\tConf$ whose projection onto $\ConfSpace$ is $\Conf$.
\end{proof}
So from both propositions, the colour of one particle in the Potts model is directly related to the connectivity of this point in the generalized Continuum Random Cluster model.
In particular the points connected "at infinity" (i.e. those in $\InfiniteCluster$) have fixed deterministic colour $1$.

For a configuration $\tConf=(\Conf,\sigma) \in \tConfSpace$, and for $\Delta\subseteq \Lambda \subseteq \R^\Dim$, we write $N_{\Delta,1}(\tConf)$ for the number of points of colour 1 inside $\Delta$.
We also write $N_{\Delta \leftrightarrow \infty}(\Conf,E)$ for the number of points in 
$\InfiniteCluster\cap{\Delta}$.
\begin{proposition} \label{propo_proportion_couleur_potts}
Assume that $\ProportionMax>1$ and that $i\not=1$ is one of the other colours with maximal proportion.
Then
\begin{align*}
\int \left(
N_{\Delta,1} - N_{\Delta,i}
\right) \
d \tSpecification^{\Intensity,\Proportion}_{\Lambda,1}
\ = \
\int
N_{\Delta \leftrightarrow \infty} \
d \RCM .
\end{align*}
\end{proposition}
\begin{proof}
From Proposition \ref{propo_projection_potts} the left hand side is
\begin{align*}
\int  
(N_{\Delta,1} - & N_{\Delta,i})
d \tSpecification^{\Intensity,\Proportion}_{\Lambda,1}
\\ & =
\int_{\ConfSpace} \int_{\EdgeSet}  
\int_{\tConfSpace}\frac{\1_{\A}(\tConf, E)}
{\FK(\A) }
 \left(
N_{\Delta,1}(\tConf) - N_{\Delta,i}(\tConf)
\right)
\lambda^{\Proportion, 1}_{\Conf,\Lambda}(\tConf)
\mu^{\phi}_{\Conf}(E) 
P_{\Lambda}^{\Intensity,\psi}(d\Conf)
\\ &=
\int_{\ConfSpace} \int_{\EdgeSet} 
\sum_{x \in \Conf_{\Delta}} 
\1_{x \in \InfiniteCluster}  \left[ \int_{\tConfSpace}
\frac{\1_{\A}(\tConf,E) }{ \mathbb{P}(\A)}
\lambda^{\Proportion, 1}_{\Conf,\Lambda}(\tConf)
 \right]
\mu^{\phi}_{\Conf}(E) 
P_{\Lambda}^{\Intensity,\psi}(d\Conf). 
\end{align*}
The integrated quantity does no longer depend on the colouring of the configuration, and from Proposition \ref{proposition_projection_gcrcm} we get the result.
\end{proof}
Intuitively if the quantity
$\int ( N_{\Delta,1} - N_{\Delta,i} )
d \tSpecification^{\Intensity,\Proportion}_{\Lambda,1}$
is bounded from below uniformly in $\Lambda$ by some $\epsilon>0$, then when $\Lambda$ goes to $\R^\Dim$ the limit of 
$ \tSpecification^{\Intensity,\Proportion}_{\Lambda,1}$ will be a Potts measure with more particles of colour 1 than any other colour.
By repeating the same with a boundary condition  of different colour, we get the existence of several different Potts measures.
From Proposition \ref{propo_proportion_couleur_potts}, to control the quantity $\int ( N_{\Delta,1} - N_{\Delta,i} )
d \tSpecification^{\Intensity,\Proportion}_{\Lambda,1}$ we need to study the connectivity in the generalized Continuum Random Cluster model
$\RCM$.
The following proposition is the key tool in proving Theorem \ref{theo_phase_transition_Potts}.
\begin{proposition}\label{propo_perco_gcrcm}
Assume that assumptions ($A1$) to ($A5$) are satisfied, and 
$\Intensity$ is large enough (depending on the parameters, but not on $\Proportion$).
Then there exists $\epsilon>0$
such that
\begin{align*}
\int N_{\Delta \leftrightarrow \infty}(\Conf,E) 
\RCM (d\Conf,dE)
\geq \epsilon
\end{align*}
for every cell $\Delta=\Delta_j$ defined after equation \eqref{eq_tempered_conf} and every $\Lambda$ finite union of cells $\Delta_{j'}$.
\end{proposition}
\subsection{Proof of Proposition \ref{propo_perco_gcrcm}}
The general idea is to use stochastic domination to compare our model to a mixed site-bond Bernoulli percolation model.
First we are decoupling the edges $E$ and constructing a probability measure 
$\bar{\mathcal{C}}^{\Intensity,\Proportion}_{\Lambda}$ the following way:
\begin{itemize}
\item
The distribution of particle positions is given by
\begin{align*}
M^{\Intensity, \Proportion}_{\Lambda}
=
\RCM
(. \times \EdgeSet).
\end{align*}
\item Given the points $\Conf$, we draw between two points $x,y \in \Conf$ such that $|x-y| \leq r_3$ an edge with probability
\begin{align}\label{eq_definition_independent_edge}
\bar{p}= \frac{1-e^{-u}}{q^2e^{-u}+1-e^{-u}}
\end{align}
where $r_3$ and $u$ come from assumption (A1).
The equation \eqref{eq_definition_independent_edge} defines on 
$\EdgeSet_{\Conf}$ the edges distribution $\bar{\mu}_{\Conf}$.
\end{itemize}
We therefore define the measure
$\bar{\mathcal{C}}^{\Intensity,\Proportion}_{\Lambda}(d\Conf,dE) 
= \bar{\mu}_{\Conf}(dE) 
M^{\Intensity, \Proportion}_{\Lambda}(d\Conf)$.
\begin{remark}
First remark that we have
$ \RCM(d\Conf,dE) = 
\mu^{\Proportion}_{\Conf,\Lambda}(dE) 
M^{\Intensity, \Proportion}_{\Lambda}(d\Conf)$ with
\begin{align*}
\mu^{\Proportion}_{\Conf,\Lambda}(dE)
\sim
\alpha_1^{|\InfiniteCluster|}
\prod_{ 
\substack{
\InfiniteCluster \not = C \subseteq \Conf_{\Lambda}
\\ \text{cluster of } (\Conf,E) }} 
\left( \sum_{i=1..\Color}\alpha_i^{|C|} \right)
\mu^{\phi}_{\Conf}(E)
\end{align*}
being the "discrete" generalized Random Cluster model.
The definition of $\mu^{\Proportion}_{\Conf,\Lambda}$ depends on $\Lambda$ only through the definition of the infinite connected component $\InfiniteCluster$.


Finally the choice $\Color^2$ in \eqref{eq_definition_independent_edge} is not optimal, but is uniform with respect to $\Proportion$.
In \cite{georgii_haggstrom}, the value $\Color$ was enough for the symmetric case.
\end{remark}
\begin{definition}
For two probability measures $\mu,\mu'$ on $\EdgeSet$, we say that $\mu'$ dominates $\mu$, written $\mu' \succeq \mu$, if $\int f d \mu' \geq \int f d \mu$ for all measurable increasing function (with respect to the natural order on $\EdgeSet$).

This notion of domination naturally extends to probability measures in 
$\ConfSpace \times \EdgeSet$.
\end{definition}
\begin{lemma}\label{lemme_domination_RCM}
Assume that assumption ($A1$) is satisfied.
Then for all $\Proportion$ and $\Conf$ we have 
$\mu^{\Proportion}_{\Conf,\Lambda} \succeq \bar{\mu}_{\Conf}$ and therefore 
$\RCM 
\succeq \bar{\mathcal{C}}^{\Intensity, \Proportion}_{\Lambda}$.
\end{lemma}
This lemma is one the principal improvement with respect to the work of Georgii and Häggström \cite{georgii_haggstrom}.
\begin{proof}
The second assertion is a direct consequence of the first one.
For the first assertion we will use the well-known \emph{Holley inequality}, see for instance \cite[Th. 2.3]{grimmett_livre_rcm}.
Let $e=\{x,y\}$ with $x,y \in \Conf$, we have
\begin{align*}
\bar{\mu}_{\Conf}(e \in E | E_{e^c})
=
\bar{\mu}_{\Conf}(e \in E)
=
\begin{cases}
\tilde{p}=\frac{1-e^{-u}}{q^2e^{-u}+1-e^{-u}}
 &\text{ if } |x-y| \leq r_3 \\
0 &\text{ if }|x-y|>r_3
\end{cases}
\end{align*}
with $\bar{\mu}_{\Conf}(e \in E | E_{e^c})$ being the probability that $e$ is an edge of $E$ conditioned on knowing all the other edges.
From easy computations we also get
\begin{align*}
\mu^{\Proportion}_{\Conf,\Lambda}(e \in E | E_{e^c})
=
\begin{cases}
1-e^{-\phi(x-y)} &\text{ if } x\leftrightarrow y \text{ in } (\Conf, E_{e^c})
\\
\left( {1 + \frac{e^{-\phi(x-y)}}{1-e^{\phi(x-y)}}
\sum_{i}  
\left( \frac{\alpha_i}{\alpha_1}\right)^{|C_y|} }\right)^{-1}
& \text{ if } x \not \leftrightarrow y \text{ and } 
x  \leftrightarrow \infty
\\
\left({1 + \frac{e^{-\phi(x-y)}}{1-e^{-\phi(x-y)}}
\frac{\sum_{i}  \alpha_i^{|C_x|} 
\sum_{j}  \alpha_j^{|C_y|} }{\sum_{i}  \alpha_i^{|C_x|+|C_y|}}}
\right)^{-1}
& \text{ if } x \not \leftrightarrow y \text{ and } 
x ,y \not \leftrightarrow \infty
\end{cases}
\end{align*}
where $\leftrightarrow$ denotes the connectivity 
in $(\Conf, E_{e^c})$ and $C_x,C_y$ are 
the connected component of $x,y$ in $(\Conf, E_{e^c})$.
Remember that the connectivity of two points can be through the imaginary point at infinity.

To apply Holley's inequality, we have to check that 
$\bar{\mu}_{\Conf}(e \in E | E_{e^c})
\leq 
\mu^{\Proportion}_{\Conf,\Lambda}(e \in E | E_{e^c})$.
We will only do it for the last expression of 
$\mu^{\Proportion}_{\Conf,\Lambda}(e \in E | E_{e^c})$.

This inequality is trivially true when $|x-y|>r_3$.
Otherwise
from assumption (A1) we have 
$\frac{e^{-\phi(x-y)}}{1-e^{-\phi(x-y)}}
\leq \frac{e^{-u}}{1-e^{-u}}$.
Furthermore
\begin{align*}
\frac{\sum_{i}  \alpha_i^{|C_x|} 
\sum_{j}  \alpha_j^{|C_y|} }{\sum_{i}  \alpha_i^{|C_x|+|C_y|}}
&=
\frac{\sum_{i}  (\alpha_i/\alpha_1)^{|C_x|} 
\sum_{j}  (\alpha_j/\alpha_1)^{|C_y|} }
{\sum_{i}  (\alpha_i/\alpha_1)^{|C_x|+|C_y|}}
\\ & \leq
\left( \frac{1}{\alpha_1}\right)^2 \frac{1}{\#^{\Proportion}_{max}}
\leq \Color^2,
\end{align*}
which implies the wanted inequality.
The other cases can be treated the same.
\end{proof}
From Lemma \ref{lemme_domination_RCM} it is enough to prove Proposition \ref{propo_perco_gcrcm} for the measure $\bar{\mathcal{C}}^{\Intensity, \Proportion}_{\Lambda}$.
This will be done by a  discretization and a comparison to the random connection model.
For this remember the definition of the cells $\Delta_j$ done just before the beginning of Section \ref{section_PPP}.
\begin{definition} \label{defi_good_cells}
Starting now we take $\delta=\frac{r_3}{\sqrt{d+3}}$, to ensure that any two points in two adjacent cells $\Delta_j,\Delta_{j'}$ are at distance at most $r_3$.
\begin{itemize}
\item We call a cell \emph{good} if it contains at least $n^*$ points forming (with the edges) a connected graph.
\item Two cells are said \emph{linked} if there exists an edge connecting two points in the two cells.
\end{itemize}
\end{definition}
This defines a correlated site-bond percolation on $\Z^\Dim$.
The next lemma states the usual percolation result for the independent site-bond percolation model.
\begin{lemma}\label{lemme_percolation_bernoulli}
Consider on $\Z^\Dim$ the Bernoulli site-bond percolation model where each site and each edge between sites at distance 1 is open with probability $p$ and closed otherwise, independently of everything else.
Let us write $Prob_p$ the probability measure associated to this model.

There exists $p_c=p_c(\Dim) \in ]0,1[$ such that for $p>p_c$,
\begin{align*}
\theta (p) 
=
Prob_p(\text{the origin is connected to infinity})
>0.
\end{align*}
\end{lemma}
The proof of this lemma is done in the appendix in Section \ref{section_appendix}.
In order to control the probability of a cell being good, we will use the following lemma.
\begin{lemma}\label{lemme_connectivity_ER}
For a positive integer $n$ and for $0<p<1$,
we consider the random graph $\mathcal{G}_{n,p}$
of $n$ vertices where each pair of vertices independently forms an edge with probability
$p$.
Then
\begin{align*}
\gamma(n,p)
=
Prob ( \mathcal{G}_{n,p} \text{ is connected} )
\underset{n \to \infty}{\longrightarrow} 1,
\end{align*}
and therefore
$\gamma(p)= \inf \{ \gamma(n,p)\ | \  n \geq 1\} >0$.
\end{lemma}
The proof is done in the appendix in Section \ref{section_appendix}.
Let us now introduce the function
\begin{align*}
h_{\Lambda}(\Conf)
= \sum_{E \in \EdgeSet_{\Conf}}\frac{\alpha_1^{|\InfiniteCluster|}
}{Z^{g}_{\Lambda}(\Proportion)}
\prod_{ 
\substack{
\InfiniteCluster \not = C \subseteq \Conf
\\ \text{cluster of } (\Conf,E) }} 
\left( \sum_{i=1..\Color}\alpha_i^{|C|} \right)
{\mu}^{\phi}_{\Conf}(E).
\end{align*}
The function $h_\Lambda$ is the probability density of 
$M^{\Intensity, \Proportion}_{\Lambda}$ with respect to 
$P_{\Lambda}^{\Intensity,\psi}$.
For fixed $\Conf$, $h_{\Lambda}(\Conf)$ is also the partition function of the discrete generalized Random Cluster model $\mu^{\Proportion}_{\Conf,\Lambda}$.
\begin{lemma}\label{lemme_domination_papangelou}
Under assumption (A2), there exists a constant $ \ConstantBis>0$ such that for every 
$\Conf \in \ConfSpace_\Lambda$ and every $x \in \Lambda$,
\begin{align*}
h_\Lambda(\Conf \cup x) \geq \ConstantBis \times h_\Lambda(\Conf).
\end{align*}
\end{lemma}
This lemma is one of the principal improvement of the initial work of Georgii and Häggström \cite{georgii_haggstrom}.
Furthermore it is the only part of the article where the finite range assumption (A2) on $\phi$ was used.
\begin{proof}
In the following, $E$ is an edge configuration between points in $\Conf$, i.e. $E \in \EdgeSet_{\Conf}$, and $E'$ is an edge configuration between $x$ and points of $\Conf$.
The union $E \cup E'$ is in $\EdgeSet_{\Conf \cup x}$. 
We will denote $C_j$, $j=1..n$ the connected components of $(\Conf,E)$, one of which can be infinite (if so it will be the first one $C_1$), which are connected together in $(\Conf \cup x, E \cup E')$.
Then we have
\begin{align*}
\frac{h_{\Lambda}(\Conf \cup x)}{h_{\Lambda}(\Conf)}
&= 
\sum_{E \in \EdgeSet_{\Conf}}
\mu^{\Proportion}_{\Conf,\Lambda}(E)
\sum_{E'}
\frac{
\alpha_1^{1+\underset{j=1..n}{\sum} |C_j|}
+ \1_{C_1\text{  finite}} 
\underset{{i=2.. \Color}}{\sum} \alpha_i^{1+\underset{j=1..n}{\sum} |C_j|}
}{\left( 
\alpha_1^{ |C_1|}
+ \1_{C_1\text{ finite}} 
\underset{{i=2.. \Color}}{\sum} \alpha_i^{ |C_1|}
\right)\underset{j=2..n}{\prod} \ \underset{i=1.. \Color}{\sum} \alpha_i^{|C_j|}
}
\mu^{\phi}_{\Conf \cup x}(E') 
\\ & \geq 
\sum_{E \in \EdgeSet_{\Conf}}
\mu^{\Proportion}_{\Conf,\Lambda}(E)
\sum_{E'} \frac{\alpha_1^{1+\underset{j=1..n}{\sum} |C_j|}}
{\underset{j=1..n}{\prod} \ \Color \alpha_1^{|C_j|}}
\ \mu^{\phi}_{\Conf \cup x}(E') 
\\ &= 
\sum_{E \in \EdgeSet_{\Conf}}
\mu^{\Proportion}_{\Conf,\Lambda}(E)
\sum_{E'}
\frac{\alpha_1 \ \mu^{\phi}_{\Conf \cup x}(E') }
{\Color^{\text{number of cc of } (\Conf,E) \text{ connected to } x}}.
\end{align*}
For a connected component to be connected to $x$ it must, from assumption (A2),  contain a point at distance less than $r_4$.
We split the closed ball $B(x,r_4)$ into a minimal number $k$ of disjoint sets $B_j, j=1..k$ of diameter less than $r_3$.
On each $B_j$, we consider the event $A_j$ that the graph $(\Conf_{B_j},E \cap \EdgeSet_{\Conf_{B_j})}$ is connected. 
The events $A_j$ are increasing and we have
\begin{align*}
\frac{h_{\Lambda}(\Conf \cup x)}{h_{\Lambda}(\Conf)}
& \geq
\frac{\alpha_1}{\Color^k}
\mu^{\Proportion}_{\Conf,\Lambda}
\left(
\bigcap_{j=1..k} A_j
\right)
\\ &\geq
\frac{\alpha_1}{\Color^k}
\bar{\mu}_{\Conf}
\left(
\bigcap_{j=1..k} A_j
\right)
\geq 
\frac{\alpha_1}{\Color^k}
\gamma ( \tilde{p})^k := \ConstantBis >0,
\end{align*}
where  the last line uses the stochastic domination of Lemma \ref{lemme_domination_RCM}, the independence of the events $A_j$ with respect to $\bar{\mu}_{\Conf}$, and where $\gamma ( \tilde{p})$ is defined in Lemma \ref{lemme_connectivity_ER}.
\end{proof}
We now have all the tools to compare 
$\bar{\mathcal{C}}^{\Intensity, \Proportion}_{\Lambda}$ to the independent site-bond Bernoulli percolation model.
Let us fix $p^*>p_c$, where $p_c$ is defined in Lemma \ref{lemme_percolation_bernoulli}.
Let us define $\lambda(n,\tilde{p})=1-(1-\tilde{p})^{n^2}$ as a lower bound for the
$\bar{\mu}_{\Conf}$-probability that there exists at least one edge between points in two neighbouring cells containing at least $n$ points each.
From Lemma \ref{lemme_connectivity_ER}, we have the existence of $n^*$ such that
\begin{align} \label{eq_borne_connectivity}
\gamma(n,\tilde{p}) \geq \sqrt{p^*} 
\quad \text{ and } \quad
\lambda(n,\tilde{p}) \geq p^* \quad \text{for all } n\geq n^*.
\end{align}
We are now in position to make clear the requirement on $r_1$ from ($A5$):
\begin{align}\label{eq_assumption_A5}
\text{(A5)} \hspace{0.5cm}
r_2 < r_3 /2\sqrt{d + 3}
\hspace{0.5cm} \text{and} \hspace{0.5cm}
(n^* -1) |B(0,r_1)| < (\delta -2r_2)^\Dim.
\end{align}
Let us define $M^{\Intensity, \Proportion}_{\Lambda,\Delta_j,\Conf}$ the conditional probability, according to $M^{\Intensity, \Proportion}_{\Lambda}$, 
of particles inside $\Delta_j$, knowing the the configuration in 
$\Lambda \setminus \Delta_j$ is $\Conf$.
\begin{lemma}\label{lemme_domination_nombre_points}
Assume that $r_1,r_2,r_3$ satisfy assumption ($A5$), 
and that assumptions (A2) and (A4) is satisfied.
Then for $\Intensity$ large enough we have
\begin{align*}
M^{\Intensity, \Proportion}_{\Lambda,\Delta_j,\Conf}
(N_{\Delta_j} \geq n^*)
\geq
\sqrt{p^*} 
\end{align*}
for all $\Lambda$ finite union of cells, for all $\Delta_j$ cells included in $\Lambda$, and for all configurations $\Conf$ on $\Lambda\setminus \Delta_j$.
\end{lemma}
The proof of this lemma is done in the appendix in Section \ref{section_appendix}.
Using this lemma, and with \eqref{eq_borne_connectivity} we obtain
\begin{align*}
M^{\Intensity, \Proportion}_{\Lambda,\Delta_j,\Conf}
(\Delta_j \text{ is good})
\geq p^*.
\end{align*}
Furthermore by construction and from \eqref{eq_borne_connectivity} the probability that two given neighbouring cells are connected, conditioned on the fact that they are good, is at least $p^*$.

Therefore by applying Holley's inequality, see Theorem 2.3 in \cite{grimmett_livre_rcm} we have
\begin{align*}
\int N_{\Delta \leftrightarrow \infty}(\Conf,E) 
\RCM (d\Conf,dE)
\geq
n^* \times \theta (p^*) := \epsilon >0.
\end{align*}
\section{Proofs of Theorem \ref{theo_existence_Potts}  and Theorem \ref{theo_phase_transition_Potts}}\label{section_proofs_theo}
Both theorems rely on the standard construction of an infinite volume Potts measure, as a limit of a stationarized finite volume Potts measures considered on a sequence of increasing boxes.
We prove that this sequence admits an accumulation point, for the topology of \emph{local convergence} using a now standard tightness tool which is the \emph{specific entropy}  developed by Georgii \cite{georgii_livre} and adapted to the continuum case by Georgii and Zessin \cite{georgii_zessin}.
Then we prove that this accumulation point is a continuum Gibbs measure, proving Theorem \ref{theo_existence_Potts}.
Finally from this construction and the Proposition \ref{propo_perco_gcrcm}, the phase transition is straightforward.
This type of construction is now standard and have been done in many articles, for the symmetric continuum Potts model \cite{georgii_haggstrom} or for other types of interaction \cite{dereudre_2009,
dereudre_drouilhet_georgii,
dereudre_houdebert,
Dereudre_Houdebert_2019_JSP_PhaseTransitionWR}.


As before we are considering $\delta=r_3/\sqrt{\Dim +3}$ as in Definition \ref{defi_good_cells}.
We consider the square box
\begin{align*}
\Lambda_n
=]-\delta (n+1/2), \delta (n+1/2)]^\Dim
\end{align*}
which is containing exactly 
$(2n+1)^\Dim$ disjoints cells 
$\Delta_j$, for $j \in L_n:= \delta \Z^\Dim \cap \Lambda_n$.

On $\tConfSpace_{\Lambda_n}$ consider the measure 
\begin{align*}
P_n (d\tConf_{\Lambda_n})
:=
\tSpecification^{\Intensity,\Proportion}_{\Lambda_n,1}(d\tConf_{\Lambda_n})
=
\1_{A_{ r_4}^1} (\tConf_{\Lambda_n})
\frac{\exp (-
\Hamiltonian_{\Lambda_n}(\tConf_{\Lambda_n}) )}
{\PartitionFunction_{\Lambda_n}(1)}
\tPoisson^{\Intensity,\Proportion}_{\Lambda_n}(d\tConf_{\Lambda_n})
\end{align*}
 defined in Proposition \ref{propo_projection_potts}.
 Finally we consider the measure
\begin{align*}
\hat{P}_n
:=
\frac{1}{(2n+1)^\Dim} \sum_{j \in L_n} \bar{P}_n \circ \tau_j^{-1},
\end{align*}
where 
$$\bar{P}_n=
\bigotimes_{k \in 2n\Z^\Dim}
P_n \circ \tau_k^{-1}.
$$
By construction the measures $\hat{P}_n$ are invariant by the translation in 
$\delta \Z^\Dim$: 
$\hat{P}_n \in \tProbaSetInvariantDelta$.
\begin{definition}\label{def_local_function_topo}
A measurable  function $f: \tConfSpace \to \R$ is said \emph{local and tame} if there exists a bounded 
$\Lambda \subseteq \R^\Dim$ and a constant $\Constant\geq 0$ such that 
\begin{align*}
f(\tConf)= f(\tConf_{\Lambda})
\ \text{ and } \
| f(\tConf) |
\leq \Constant ( 1 + N_{\Lambda}(\tConf) )
\end{align*}
for all configurations $\tConf \in \tConfSpace$.

A sequence of measures $\nu_n$ converge to $\nu$ in the \emph{local convergence topology} if
$\int f \ d\nu_n \to \int f \ d\nu$ for all local and tame functions $f$.
\end{definition}
\begin{proposition}\label{propo_existene_cluster_point_Potts}
The sequence $(\hat{P}_n)$ admits a cluster point $\hat{P}$ with respect to the local convergence topology.
This cluster point is invariant under the translation 
$\tau_x$, $x\in \delta \Z^\Dim$, and it is a Potts measure:
$\hat{P} \in \GibbsPotts_{\theta_\delta}(\Intensity, \Proportion)$.
\end{proposition}
\begin{remark}
In the following, to lighten the notation, we will avoid to take a subsequence and assume that $(\hat{P}_n)$ converges to $\hat{P}$.
\end{remark}
We will first admit this proposition and conclude the proofs of Theorem \ref{theo_existence_Potts} and Theorem \ref{theo_phase_transition_Potts}.
\subsection{Proof of Theorem \ref{theo_existence_Potts}}
Proposition \ref{propo_existene_cluster_point_Potts} is not enough to conclude directly the proof of Theorem \ref{theo_existence_Potts}, since the measure $\hat{P}$ is not invariant under all translation of $\R^\Dim$.
But considering the measure
\begin{align*}
\widetilde{P}
:=
\frac{1}{\delta^\Dim}
\int_{]-\delta/2,\delta/2]^\Dim}
\hat{P} \circ \tau_x^{-1} dx,
\end{align*}
we obtain a measure which is by construction invariant under all translation of 
$\R^\Dim$.
This measure satisfies $0<\PartitionFunction_{\Lambda}(.)<\infty$ 
$\widetilde{P}$-almost surely for every bounded $\Lambda$, and from the translation invariance of the interaction, we obtain 
\begin{align*}
\int & \int  f(\tConf'_{\Lambda}\tConf_{\Lambda^c})
\tSpecification^{\Intensity,\Proportion}_{\Lambda,\tConf}(d\tConf'_{\Lambda})
\widetilde{P}(d\tConf)
\\ & =
\underset{\left] -\frac{\delta}{2}, \frac{\delta}{2} \right]^\Dim}{\int}
\int \int
f(\tConf'_{\Lambda} \tau_x(\tConf)_{\Lambda^c})
\frac{e^{- \Hamiltonian_{\Lambda} 
(\tConf'_{\Lambda} \tau_x(\tConf)_{\Lambda^c})}}
{\delta^\Dim   \PartitionFunction_{\Lambda}(\tau_x(\tConf))}
\Poisson^{\Intensity, \Proportion}_{\Lambda}(d \tConf')
\hat{P}(d\tConf)   dx
\\ & =
\underset{\left] -\frac{\delta}{2}, \frac{\delta}{2} \right]^\Dim}{\int}
\int \int
f \circ \tau_x \left(
\tConf'_{\tau_x^{-1}(\Lambda)} \tConf_{\tau_x^{-1}(\Lambda)^c}
\right)
\frac{e^{- \Hamiltonian_{\Lambda} \circ \tau_x
\left(
\tConf'_{\tau_x^{-1}(\Lambda)} \tConf_{\tau_x^{-1}(\Lambda)^c}
\right)}}
{\delta^\Dim   \PartitionFunction_{\Lambda}(\tau_x(\tConf))}
d\Poisson^{\Intensity, \Proportion}
d\hat{P} dx
\\ & =
\frac{1}{\delta^\Dim}
\underset{\left] -\frac{\delta}{2}, \frac{\delta}{2} \right]^\Dim}{\int}
\int
f \circ \tau_x 
\ d \hat{P} \ dx
=
\int
f \ 
d \widetilde{P}
\end{align*}
and therefore $\widetilde{P}$ is a Potts measure invariant under 
all translation of $\R^\Dim$.

Since the set of translation invariant Potts measures is a convex set with extremal elements being the ergodic Potts measures, see \cite{georgii_livre}, the theorem is proved.
\subsection{Proof of Theorem \ref{theo_phase_transition_Potts}}
Let us consider $\Delta=\Delta_0=\Lambda_0=]-\delta /2, \delta /2]^\Dim$.
Let $i \not = 1$ an other colour with maximal proportion (i.e. $\alpha_1=\alpha_i$) , we obtain from Proposition \ref{propo_proportion_couleur_potts} that
\begin{align*}
\int ( N_{\Delta, 1} - N_{\Delta, i}) d \hat{P}_n
&=
\frac{1}{(2n+1)^\Dim}
\sum_{j \in L_n}
\int ( N_{\Delta_j, 1} - N_{\Delta_j, i}) d {P}_n
\\ &=
\frac{1}{(2n+1)^\Dim}
\sum_{j \in L_n}
N_{\Delta_j \leftrightarrow \infty} \ d \RCMn.
\end{align*}
Finally we obtain from Proposition \ref{propo_perco_gcrcm} that for $\Intensity$ large enough (but independent of $\Proportion$) that
\begin{align*}
\int ( N_{\Delta, 1} - N_{\Delta, i}) d \hat{P}_n
\geq \epsilon>0.
\end{align*}
Since the integrated function is local and tame, the same bound is valid for 
the probability measure $\hat{P}$.

For the measure $\widetilde{P}$, even if the translated 
$\tau_x(\Delta)$, $x \in ]-\delta /2, \delta /2]^\Dim$ is not a cell as defined before, using the translation invariance of $\hat{P}$ by vectors $j \in (\delta \Z)^\Dim$, one can translate back $\tau_x(\Delta)$ into $\Delta$, which proves that
\begin{align*}
\int ( N_{\Delta, 1} - N_{\Delta, i}) d \widetilde{P}
\geq \epsilon>0,
\end{align*}
see Figure \ref{figure_translation}.

\begin{figure}[h]\label{figure_translation}
\includegraphics[width=6.5cm]{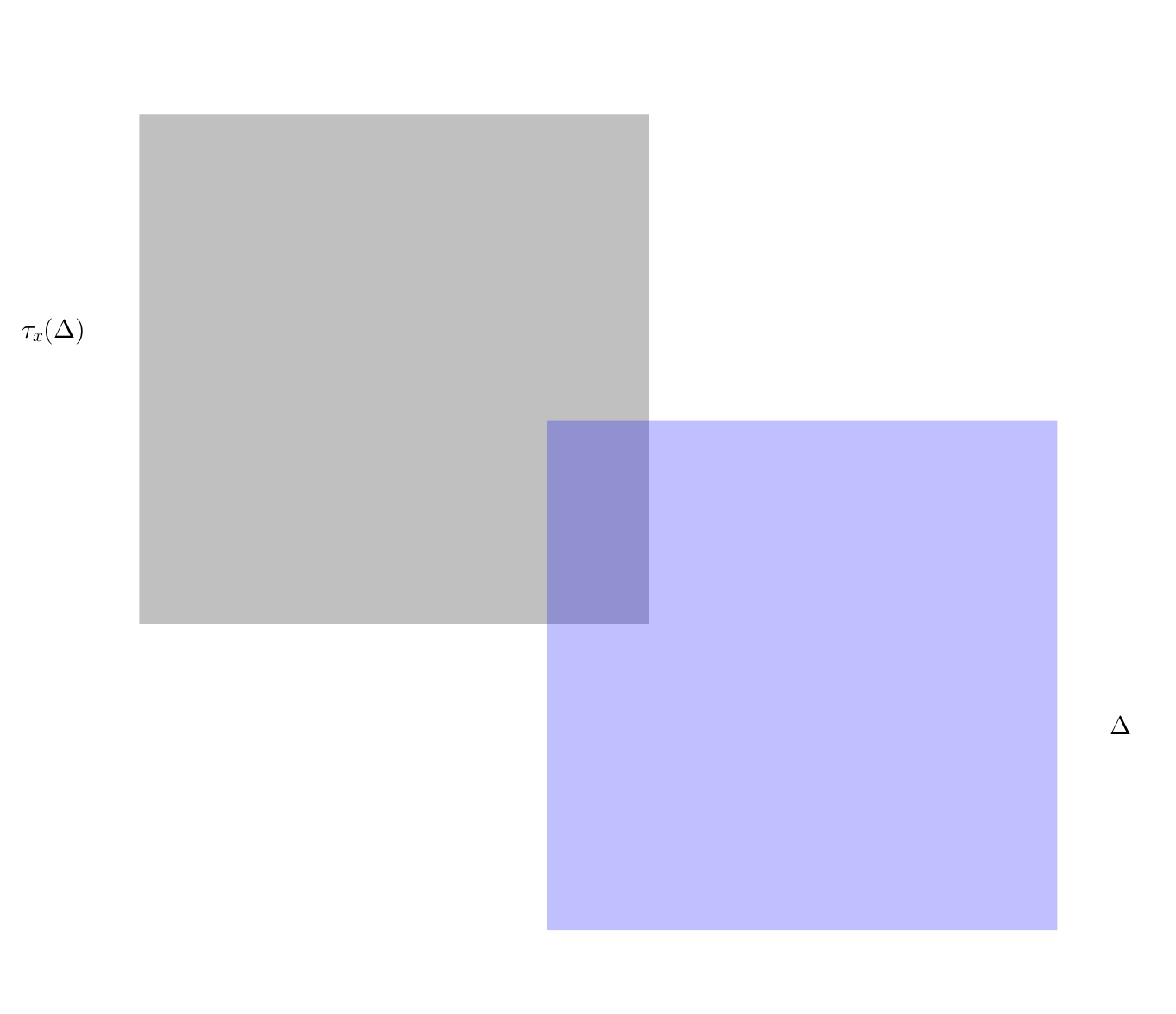}
\includegraphics[width=6.5cm]{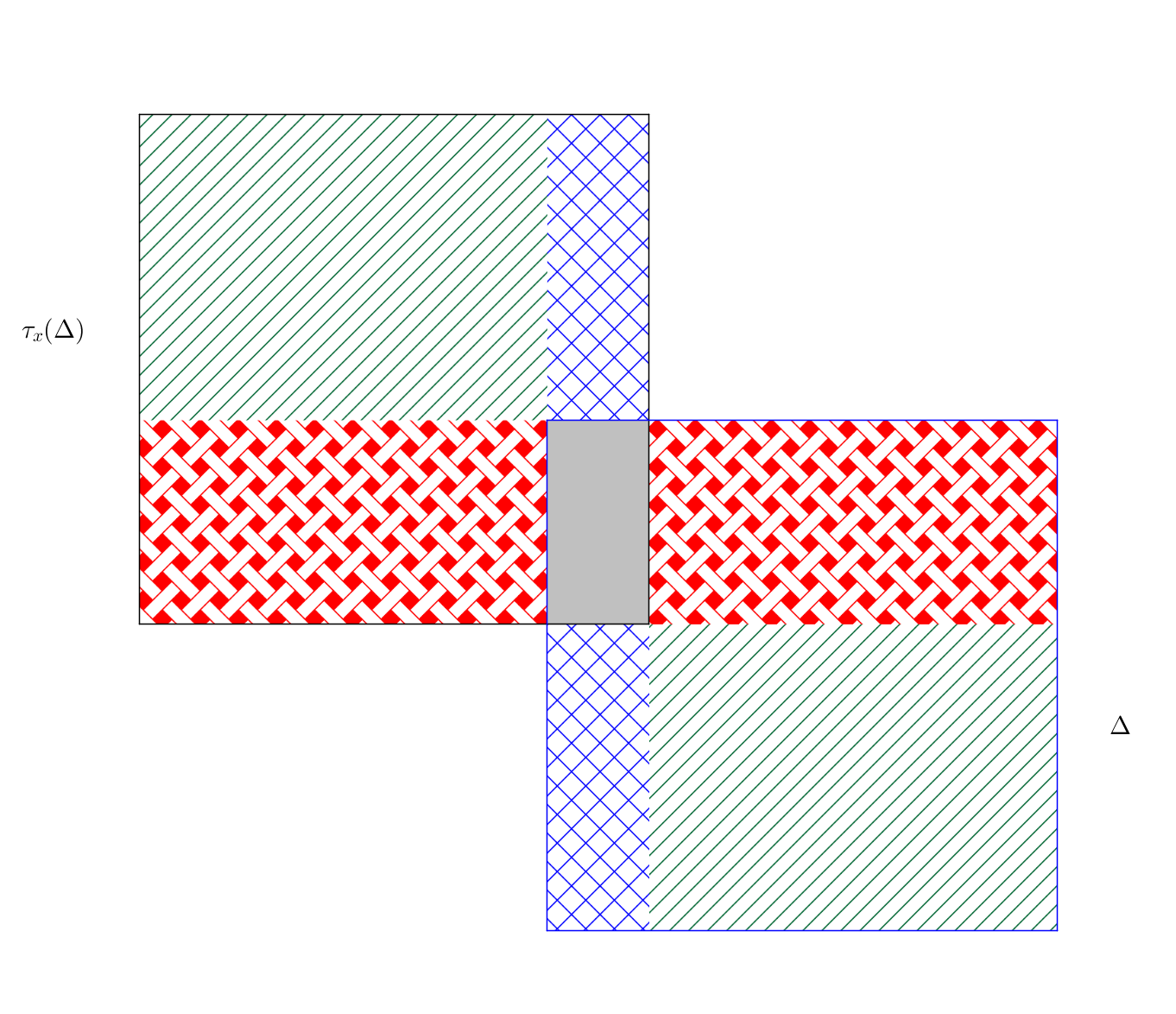}
\caption{Left: The cell $\Delta$ and its translation $\tau_x(\delta)$. Right: how to translate back piece of $\tau_x(\Delta)$ to obtain again $\Delta$.}
\end{figure}

So the probability measure $\widetilde{P}$ is a Potts measure with on average more particles of colour $1$ that any other colours.
By repeating the same construction for every colour $i$ with maximal proportion, we obtain $\ProportionMax$ Potts measures invariant by all translations in $\R^\Dim$ and which are different, since one different colour dominates the others for each measure.

From the ergodic decomposition of translation invariant (by a vector of $\R^\Dim$) Potts measures, see \cite{georgii_livre}, 
it is clear that one can find at least $\ProportionMax$ different ergodic Potts measures.
The theorem is proved.

\subsection{Proof of Proposition \ref{propo_existene_cluster_point_Potts}}
This type of construction is classical.
It was done for instance for the symmetric Potts model in \cite{georgii_haggstrom}, for the quermass-interaction model in \cite{dereudre_2009}, and for many other cases \cite{dereudre_drouilhet_georgii,dereudre_houdebert,
Dereudre_Houdebert_2019_JSP_PhaseTransitionWR}.
The first step is to construct a good candidate.
This is done using the specific entropy as a tightness tool.
Then one has to prove that this good candidate is indeed a Potts measure, which is done by approximation of the interaction.

$\bullet$ {\bf Step 1}: Construction of a good candidate
\begin{definition}
For a measure $P \in \tProbaSetInvariantDelta$
with finite first moment, meaning that $\int N_{\Delta_0} d P<\infty$, 
we define the specific entropy
\begin{align}\label{eq_def_specific_entropy}
\Entropy(P)
:=
\lim_{n \to \infty}
\frac{1}{|\Lambda_n|}
\Entropy_{\Lambda_n}(P| \Poisson^{\Intensity, \Proportion}),
\end{align}
with
$\Entropy_{\Lambda_n}(P| \Poisson^{\Intensity, \Proportion})
= \int \log \left(
\frac{d P_{\Lambda_n}}{d \Poisson^{\Intensity, \Proportion}_{\Lambda_n}}
(\tConf_{\Lambda_n}) \right)  P(d\tConf)$, or 
$+\infty$ if $P_{\Lambda_n}$ 
(the restriction of $P$ in $\tConfSpace_{\Lambda_n}$) is not absolutely continuous with respect to $\Poisson^{\Intensity, \Proportion}_{\Lambda_n}$.
\end{definition}
The convergence in \eqref{eq_def_specific_entropy} is proved in \cite{georgii_zessin}. The next proposition, also proved in \cite{georgii_zessin}, stated the tightness of the level sets of the specific entropy.
\begin{proposition}
On the set of probability measure $P \in \tProbaSetInvariantDelta$ with finite first moment, the specific entropy is \emph{affine} and \emph{upper semi-continuous}.
Furthermore for all $\Constant \geq 0$, the level set $\Entropy(P)\leq \Constant$ is compact and sequentially compact with respect to the local convergence topology.
\end{proposition}
So from this proposition it is enough to prove that the specific entropy 
of the sequence $(\hat{P}_n)$ is uniformly bounded.
It is clear that each $\hat{P}_n$ has finite first moment (but it is not clear yet that one can find an uniform bound).
From the fact that the specific entropy is affine, we obtain that 
\begin{align*}
\Entropy (\hat{P}_n)
&= \frac{1}{| \Lambda_n |} 
\Entropy_{\Lambda_n}  \left( 
P_n | \Poisson^{\Intensity, \Proportion}_{\Lambda_n}
\right)
=
\frac{-\log (\PartitionFunction_{\Lambda_n}(1))
- \int_{\tConfSpace} H_{\Lambda_n}(\tConf_{\Lambda_n})
P_n ( d \tConf_{\Lambda_n})}{| \Lambda_n |} 
\end{align*}
For the assumptions (A1) and (A3) we have
\begin{align}\label{eq_borne_energy}
H_{\Lambda_n}(\tConf_{\Lambda_n})
\geq
H^\Psi_{\Lambda_n}(\Conf_{\Lambda_n})
\geq
\sum_{j \in L_n}a N_{\Delta_j}(\Conf)^2 - b  N_{\Delta_j}(\Conf)
\geq -(2n+1)^\Dim \Constant,
\end{align}
where $\Constant=B^2/4A$.
Furthermore from standard computation we obtain
\begin{align*}
\PartitionFunction_{\Lambda_n}(1)
\geq e^{- \Intensity |\Lambda_n|}.
\end{align*}
Therefore we obtain
\begin{align*}
\Entropy(\hat{P}_n)
\leq
c_1
\end{align*}
for a positive finite constant $c_1$ independent of $n$.
Hence we have the existence of a cluster point $P$ with respect to the local convergence topology.
In the following we omit to take a subsequence to lighten the notation.
\begin{remark}
To obtain \eqref{eq_borne_energy} we used the superstability and regularity from assumption (A3).
In the case when the potential $\psi$ is non-negative, the Hamiltonian is non-negative and we obtain again the existence of a constance $c_1$.
\end{remark}

$\bullet$ {\bf Step 2}: the partition function $\PartitionFunction_\Lambda$ is $\hat{P}$-a.s non degenerate.
\newline
To prove this, let us first prove that $\hat{P}$ has finite second moment.
\begin{lemma}\label{lemma_second_moment_candidate}
Under assumptions (A1) and (A3), we have for all $n$
$$
\int N_{\Delta_0}^2 \ d \hat{P}
< \infty
\hspace{1cm}
\text{and}
\hspace{1cm}
\int N_{\Delta_0}^2 \ d \hat{P}_n
< \infty.
$$
\end{lemma}
\begin{proof}
Let us now consider the second moment of $\hat{P}_n$:
\begin{align*}
\int_{\tConfSpace}N_{\Delta_0}(\tConf)^2\hat{P}_n(d\tConf)
&=
\frac{1}{(2n+1)^\Dim}
\int_{\tConfSpace}
\sum_{j \in L_n} N_{\Delta_j}(\Conf)^2 P_n(d\tConf)
\\ & \leq
\frac{2}{a(2n+1)^\Dim}
\int_{\tConfSpace} H_{\Lambda_n}(\tConf_{\Lambda_n})
P_n(d\tConf)
\\ & 
 \hspace{1cm} +
\frac{1}{(2n+1)^\Dim}
\int_{\tConfSpace}
\sum_{j \in L_n}\left(
\frac{2b}{a} N_{\Delta_j}(\Conf)-  N_{\Delta_j}(\Conf)^2
\right)  P_n(d\tConf),
\end{align*}
where the last inequality is a consequence of assumptions (A1) and (A3), used as in
\eqref{eq_borne_energy}.
From the non-negativity of the local entropy, we have that
\begin{align*}
\int_{\tConfSpace} H_{\Lambda_n}(\tConf_{\Lambda_n})
P_n(d\tConf)
\leq - \log ((\PartitionFunction_{\Lambda_n}(1))
\leq \Intensity |\Lambda_n|.
\end{align*}
Furthermore there exists a constant $\Constant \geq 0$ such that
\begin{align*}
\int_{\tConfSpace}
\sum_{j \in L_n}\left(
\frac{2b}{a} N_{\Delta_j}(\Conf)- N_{\Delta_j}(\Conf)^2
\right)  P_n(d\tConf)
\leq \Constant (2n+1)^\Dim.
\end{align*}
Putting everything together we obtain 
\begin{align*}
\int_{\tConfSpace} N_{\Delta_0}(\Conf)^2 \hat{P}_n(d\tConf)
\leq
\tilde{\Constant},
\end{align*}
where $\tilde{\Constant}<\infty$ is independent of $n$.

Now the function $N_{\Delta_0}^2$ is local but not tame.
However this is the monotone limit of local and tame functions, which is enough to conclude that
\begin{align*}
\int_{\tConfSpace} N_{\Delta_0}(\Conf)^2 \hat{P}(d\tConf)
\leq
\tilde{\Constant}.
\end{align*}

In the case where $\psi\geq 0$, one can prove immediately from the stochastic domination result of Georgii and Küneth \cite{georgii_kuneth} that the measure $P_n$ is stochastically dominated by $\Poisson^{\Intensity, \Proportion}_{\Lambda_n}$ and therefore the uniform bound is straightforward.
\end{proof}
Now let us  define on $\tConfSpace$ the space of tempered configuration 
\begin{align}
\label{eq_tempered_conf}
\TemperedConfSpace=\left\{ \tConf\in\tConfSpace ,\ \sup_{n \geq 1 } \ \frac{1}{n^{\Dim}}
   \sum_{|i|\leq n}  N_{\Delta_j}^2(\tConf) < \infty \right\}
\end{align}
\begin{lemma} \label{lemme_tempered_non-degenere}
Under assumptions (A1) and (A3), and
for all $\tConf \in \TemperedConfSpace$, the partition function is non-degenerate:
\begin{align*}
0< \PartitionFunction_{\Lambda}(\tConf)<\infty.
\end{align*}
Furthermore we have $\hat{P} (\TemperedConfSpace)=1$.
\end{lemma}
This lemma is proved in the appendix in Section \ref{section_appendix}.

$\bullet$ {\bf Step 3}: the measure $\hat{P}$ satisfies the DLR equations.
\newline
It is enough to proves the DLR equations only for the $\Lambda_n$, and starting now we fix $\Lambda=\Lambda_{n_0}$ for a fixed $n_0\in \N$.
Let us consider $f$ a measurable function bounded by one, which we can assume without loss of generality that is is local, i.e 
$f(\tConf)=f(\tConf_{\Lambda_{n_1}})$ for a fixed $n_1$.
we are interested in proving that the following quantity
\begin{align*}
\kappa  :=
\bigg\lvert
\int f d \hat{P}
-
\int \int
f(\tConf'_\Lambda \tConf_{\Lambda^c})
\tSpecification^{\Intensity,\Proportion}_{\Lambda,\tConf}(d\tConf'_{\Lambda})
\hat{P} (d \tConf)
\bigg\rvert
\end{align*}
is small.
The first issue is that the probability measures $\hat{P}_n$ do not satisfy the DLR equations, except in the particular case of $\Lambda \subseteq \Delta_0$.
We are introducing the new sequence of measures 
\begin{align*}
\hat{P}_n^{\Lambda}
:=
\frac{1}{(2n+1)^\Dim}
\sum_{j \in L_N} \1_{\Lambda \subseteq \tau_j(\Lambda_n)}
 P_n \circ \tau_j^{-1}.
\end{align*}
Those are not probability measure, but from the following lemma they satisfy the DLR($\Lambda$) equation and are converging to $\hat{P}$.
\begin{lemma}\label{lemme_dlr_suite}
Each $\hat{P}_n^{\Lambda}$ satisfies the DLR($\Lambda$) equation.
Furthermore if (A1) and (A3) are satisfied,  for all local and tame functions $f$, we have
\begin{align*}
\left\lvert
\int f d \hat{P}_n^{\Lambda} - \int f d \hat{P}_n
\right\rvert 
\underset{n \to \infty}{\longrightarrow} 0.
\end{align*}
\end{lemma}
\begin{proof}
The first point is a consequence of the compatibility of the Gibbs kernels and the translation invariance of the interaction.
The second point has been treated in \cite{dereudre_2009} for the quermass interaction model or in \cite{dereudre_houdebert} for the Continuum Random Cluster model, and we are omitting the proof here.
\end{proof}
\begin{lemma}\label{lemme_local_approximation}
Under assumptions (A1), (A2) and (A3) we have for all $N$ large enough and for all $n$
\begin{align*}
\bigg\lvert
 \int \int
f(\tConf'_\Lambda \tConf_{\Lambda^c})
\tSpecification^{\Intensity,\Proportion}_{\Lambda,\tConf}(d\tConf'_{\Lambda})
\hat{P} (d \tConf)
-
 \int \int
f(\tConf'_\Lambda \tConf_{\Lambda^c})
\tSpecification^{\Intensity,\Proportion}_{\Lambda,\tConf_{\Lambda_N}}(d\tConf'_{\Lambda}) \hat{P} (d \tConf)
\bigg\rvert
\leq \epsilon
\end{align*}
and
\begin{align*}
\bigg\lvert
 \int \int
f(\tConf'_\Lambda \tConf_{\Lambda^c})
\tSpecification^{\Intensity,\Proportion}_{\Lambda,\tConf}(d\tConf'_{\Lambda})
\hat{P}_n^\Lambda (d \tConf)
-
 \int \int
f(\tConf'_\Lambda \tConf_{\Lambda^c})
\tSpecification^{\Intensity,\Proportion}_{\Lambda,\tConf_{\Lambda_N}}
(d\tConf'_{\Lambda})
\hat{P} _n^\Lambda(d \tConf)
\bigg\rvert
\leq \epsilon
\end{align*}
\end{lemma}
The proof is classical and is done in the appendix in Section \ref{section_appendix}.
Let us now conclude the proof of Proposition \ref{propo_existene_cluster_point_Potts}.
Let us fix $\epsilon>0$.
From Lemma \ref{lemme_local_approximation} there is $N$ large enough such that
\begin{align*}
\kappa  
\leq
\epsilon +
\bigg\lvert
\int f d \hat{P}
-
\int \int
f(\tConf'_\Lambda \tConf_{\Lambda^c})
\tSpecification^{\Intensity,\Proportion}_{\Lambda,\tConf_{\Lambda_N}}
(d\tConf'_{\Lambda})   \hat{P} (d \tConf)
\bigg\rvert.
\end{align*}
Now from the second point of Lemme \ref{lemme_dlr_suite} we obtain for $n$ large enough
\begin{align*}
\kappa  
\leq
2\epsilon +
\bigg\lvert
\int f d \hat{P}_n^{\Lambda}
-
\int \int
f(\tConf'_\Lambda \tConf_{\Lambda^c})
\tSpecification^{\Intensity,\Proportion}_{\Lambda,\tConf_{\Lambda_N}}
(d\tConf'_{\Lambda})   \hat{P}_n^{\Lambda}(d \tConf)
\bigg\rvert,
\end{align*}
and applying again Lemma \ref{lemme_local_approximation} 
and the first point of Lemma \ref{lemme_dlr_suite},
\begin{align*}
\kappa  
\leq
3 \epsilon +
\bigg\lvert
\int f d \hat{P}_n^{\Lambda}
-
\int \int
f(\tConf'_\Lambda \tConf_{\Lambda^c})
\tSpecification^{\Intensity,\Proportion}_{\Lambda,\tConf}
(d\tConf'_{\Lambda})   \hat{P}_n^{\Lambda}(d \tConf)
\bigg\rvert
= 3 \epsilon,
\end{align*}
and the proof is concluded.
\newpage
\section{Appendix: proof of the intermediary lemmas}\label{section_appendix}
\subsection{Proof of Lemma \ref{lemme_percolation_bernoulli}}
The mixed site-bound Bernoulli percolation model is clearly monotone in 
the parameter $p$, which gives the existence of $p_c$.
It remains to prove that $p_c$ is in $]0,1[$.
It is clear that $p_c>0$,  since the absence of percolation in the site Bernoulli percolation model of parameter $p$ implies the same for the mixed site-bond model of parameter $p$.

The second inequality comes from the observation that for any graph $G$, the site percolation threshold is not greater than the bond percolation threshold (both for the Bernoulli model).
Hence $p_c$ is smaller than the square root of the site Bernoulli percolation model.
\subsection{Proof of Lemma \ref{lemme_connectivity_ER}}
Let us denote $x_1,\dots , x_n$ the $n$ points of the graph.
By considering the event that $x_1$ is connected to every other points by a path of length exactly 2, we obtain
\begin{align*}
\gamma(n,p)
\geq
1- (n-1) (1-p^2)^{n-2},
\end{align*}
which proves the result.
\subsection{Proof of Lemma \ref{lemme_domination_nombre_points}}
It is easy to see that
\begin{align*}
M^{\Intensity, \Proportion}_{\Lambda,\Delta_j,\Conf} (d \Conf')
=
\frac{h_{\Lambda}(\Conf' \Conf)}{Z_{\Lambda,\Delta_j,\Conf}}
\exp \left( -
\Hamiltonian^{\psi}_{\Delta_j}(\Conf' \Conf) 
\right)
\Poisson^{\Intensity}_{\Delta_j} (d \Conf'),
\end{align*}
where $Z_{\Lambda,\Delta_j,\Conf}$ is the corresponding partition function.
Therefore we obtain for $n\geq 0$ that
\begin{align*}
\frac{M^{\Intensity, \Proportion}_{\Lambda,\Delta_j,\Conf} (N_{\Delta_j}=n+1)}
{M^{\Intensity, \Proportion}_{\Lambda,\Delta_j,\Conf} (N_{\Delta_j}=n)}
=
\frac{z}{n+1}\int g_i(\Conf' \Conf)
M^{\Intensity, \Proportion}_{\Lambda,\Delta_j,\Conf} (d\Conf' | N_{\Delta_j}=n)
\end{align*}
with
\begin{align*}
g_i(\Conf' \Conf)
&=
\int_{\Delta_j}
\exp \left(
-\sum_{y\in \Conf' \Conf} \psi (x-y)
\right)
\frac{h_{\Lambda}(\Conf' \Conf \cup x)}{h_{\Lambda}(\Conf' \Conf)}
dx
\\ & \geq
l \int_{\Delta_i}
\exp \left(
-\sum_{y\in \Conf' \Conf} \psi (x-y)
\right)
dx,
\end{align*}
with the last inequality coming from Lemma \ref{lemme_domination_papangelou}.
Consider now the reduce cell $\Delta_j^0$ obtain from $\Delta_j$ by removing a boundary layer of width $r_2$.
By assumption ($A5$), $\Delta_j^0$ has positive volume.
Then
\begin{align*}
g_i(\Conf' \Conf)
&\geq
l \int_{\Delta_j^0}
\exp \left(
-\sum_{y\in \Conf' \Conf} \psi (x-y)
\right)
dx
\\ & \geq
l \int_{\Delta_j^0}
\exp \left(
-\sum_{y\in \Conf'} \psi (x-y)
\right)
dx,
\end{align*}
where the last inequality comes from assumption ($A4$).
This last bound is independent of the boundary condition $\Conf$.
The next estimates is directly taken from \cite{georgii_haggstrom}, and goes back originally to Dobrushin and Minlos.
Let
\begin{align*}
\Delta_{\Conf'}
=  \{ 
x \in \Delta_j^0 , \ |x-y| \geq r_1 \text{ for all } y \in \Conf'
\}.
\end{align*}
Assume by contradiction that $N_{\Delta_j}(\Conf')<n^*$. Then
\begin{align*}
| \Delta_{\Conf'} |
\geq 
| \Delta_j^0 | - (n^* -1)|B(0,r_1)|
:= v*,
\end{align*}
and $v^*$ is positive thanks to assumption (A5), see equation \eqref{eq_assumption_A5}.
Furthermore, applying Markov's inequality to the Lebesgue measure, we obtain for all $\Constant>0$
\begin{align*}
|\{
x \in \Delta_{\Conf'} ,  \sum_{y \in \Conf'} \psi (x-y) \geq \Constant
\}|
& \leq
\frac{1}{\Constant} \sum_{y \in \Conf'} \int_{\Delta_{\Conf'}} \psi^+(x-y) dx
\\ & \leq
\frac{n^*-1}{\Constant} \int_{|x|\geq r_1} \psi^+(x) dx
:= \frac{b(n^*,r_1)}{\Constant},
\end{align*}
with $b(n^*,r_1)<\infty$ thanks to assumption (A4).
Adding everything together we obtain when $N_{\Delta_j}(\Conf')<n^*$  that
\begin{align*}
g_i(\Conf' \Conf) \geq l e^{-\Constant} \left( v^* - \frac{b(n^*,r_1)}{\Constant} \right).
\end{align*}
By choosing $\Constant$ large enough, there exists a constant $\tilde{l}$ such that
$g_i(\Conf' \Conf) \geq \tilde{l}n^*$.
Hence
\begin{align*}
M^{\Intensity, \Proportion}_{\Lambda,\Delta_j,\Conf} (N_{\Delta_j}<n*)
& =
\sum_{n=0}^{n^*-1}
M^{\Intensity, \Proportion}_{\Lambda,\Delta_j,\Conf} (N_{\Delta_j}=n)
\\ &=
\sum_{n=0}^{n^*-1}
M^{\Intensity, \Proportion}_{\Lambda,\Delta_j,\Conf} (N_{\Delta_j}=n^*)
\prod_{k=n}^{n^*-1}
\frac{M^{\Intensity, \Proportion}_{\Lambda,\Delta_j,\Conf} (N_{\Delta_j}=k)}
{M^{\Intensity, \Proportion}_{\Lambda,\Delta_j,\Conf} (N_{\Delta_j}=k+1)}
\\ & \leq
\sum_{n=0}^{n^*-1}
\left(
\frac{1}{\tilde{l} \Intensity}
\right)^{n^*-n}
\leq 
\frac{1}{\tilde{l} \Intensity -1},
\end{align*}
and this quantity goes to $0$ when $\Intensity$ goes to infinity.
Therefore the result is proved.

\subsection{Proof of Lemma \ref{lemme_tempered_non-degenere}}
For the first point, let us consider  $\Lambda\subset \R^\Dim$ bounded, $\tConf \in \TemperedConfSpace$ and $\tConf' \in \tConf$.
From standard computation we obtain
$
\PartitionFunction_{\Lambda}(\tConf) \geq exp(-\Intensity |\Lambda|)>0.
$
Using assumption (A1) we have
\begin{align*}
\Hamiltonian_{\Lambda} (\tConf'_{\Lambda}\tConf_{\Lambda^c})
&\geq
\Hamiltonian^\psi_{\Lambda} (\tConf'_{\Lambda}\tConf_{\Lambda^c}).
\end{align*}
Now let us consider assumption (A3).
In the case $\psi \geq 0$ we immediately obtain 
than the Hamiltonian is non negative, and hence 
$\PartitionFunction_{\Lambda}(\tConf)\leq 1 < \infty$.
In the other case we obtain from assumption (A3) that
\begin{align*}
\Hamiltonian_{\Lambda} (\tConf'_{\Lambda}\tConf_{\Lambda^c})
&\geq
\sum_{j'} \left[
a N_{\Delta_{j'}}(\tConf')^2 
- \left(b + \sum_j \psi_{\delta^{-1}|j-j'|}  N_{\Delta_j}(\tConf) \right)
 N_{\Delta_{j'}}(\tConf')    \right]
+ error,
\end{align*}
where the sum is over $j'$ such that 
$\Delta_{j'} \cap \Lambda \not =0$ and $j$ 
such that $\Delta_j \cap \Lambda =0$.
The error term comes from the fact that we did not take into consideration points 
$x\in\tConf'_\Lambda$ and $y\in\tConf_\Lambda^c$ 
which are in the same cell $\Delta_j$.
But since the superstability and regularity does not depend on the choice of the discretization, we can assume without loss of generality that this error term is null, which is to say that $\Lambda$ is exactly the union of a finite number of cells $\Delta_j$.
Now from the fact that $\tConf$ is tempered, we have the existence of a constant $\Constant\geq 0$ such that
\begin{align*}
\sum_k \psi_{\delta^{-1}|j-k|}  N_{\Delta_k}(\tConf)
\leq
\sum_{n \in \N} \psi_n \sum_{k, \delta^{-1}|j-k|=n} N_{\Delta_k}(\tConf)
\leq
\Constant \sum_{n \in \N} n^{\Dim-1}\psi_n
:= B'<\infty
\end{align*}
and therefore $\Hamiltonian_{\Lambda} (\tConf'_{\Lambda}\tConf_{\Lambda^c})$ is bounded from below uniformly in $\tConf'$, and so
$$\PartitionFunction_{\Lambda}(\tConf) < \infty.$$

From construction and from Lemma \ref{lemma_second_moment_candidate}, the measure $\hat{P}$ is invariant under the translation of $(\delta \Z)^\Dim$ and satisfies $\int N_{\Delta_0}^2 d \hat{P}<\infty$.
Therefore, from the ergodic theorem, see
 \cite{Nguyen_Zessin_1979_ErgodicTheoremSpatialProcesses}, the sequence of  random variables
 \begin{align*}
n \mapsto \frac{1}{n^\Dim}
 \sum_{|i| \leq n}  N_{\Delta_0} (\tConf)^2
 \end{align*}
 converges $\hat{P}$ almost surely towards a finite random variable.
 The result is proved.
\subsection{Proof of Lemma \ref{lemme_local_approximation}}
Consider $\tConf \in \TemperedConfSpace$. 
Then
\begin{align*}
c(\tConf) :=
&\bigg\lvert
 \int 
f(\tConf'_\Lambda \tConf_{\Lambda^c})
\tSpecification^{\Intensity,\Proportion}_{\Lambda,\tConf}(d\tConf'_{\Lambda})
-
 \int 
f(\tConf'_\Lambda \tConf_{\Lambda^c})
\tSpecification^{\Intensity,\Proportion}_{\Lambda,\tConf_{\Lambda_N}}
(d\tConf'_{\Lambda})
\bigg\rvert
\\ & \quad \leq 
\int
\bigg\lvert
\frac{e^{-\Hamiltonian_{\Lambda}(\tConf'_{\Lambda}\tConf_{\Lambda^c})}}
{\PartitionFunction_{\Lambda}(\tConf)}
-
\frac{e^{-\Hamiltonian_{\Lambda}
(\tConf'_{\Lambda}\tConf_{\Lambda_N \setminus \Lambda})}}
{\PartitionFunction_{\Lambda}(\tConf_{\Lambda_N})}
\bigg\rvert
\Poisson_{\Lambda}^{\Intensity, \Proportion}(d \tConf')
\\ & \quad \leq 
\int
\bigg\lvert
\frac{e^{-\Hamiltonian_{\Lambda}(\tConf'_{\Lambda}\tConf_{\Lambda^c})}
- e^{-\Hamiltonian_{\Lambda}
(\tConf'_{\Lambda}\tConf_{\Lambda_N \setminus \Lambda})}}
{\PartitionFunction_{\Lambda}(\tConf)}
\bigg\rvert
\Poisson_{\Lambda}^{\Intensity, \Proportion}(d \tConf')
+
\bigg\lvert
\frac{\PartitionFunction_{\Lambda}(\tConf_{\Lambda_N})}
{\PartitionFunction_{\Lambda}(\tConf)}
-
1
\bigg\rvert
\\ &  \quad \leq
2   \int 
\bigg\lvert
\frac{e^{-\Hamiltonian_{\Lambda}(\tConf'_{\Lambda}\tConf_{\Lambda^c})}
- e^{-\Hamiltonian_{\Lambda}
(\tConf'_{\Lambda}\tConf_{\Lambda_N \setminus \Lambda})}}
{\PartitionFunction_{\Lambda}(\tConf)}
\bigg\rvert
\Poisson_{\Lambda}^{\Intensity, \Proportion}(d \tConf').
\end{align*}
Now using the mean value theorem and by considering $N$ large enough we obtain from assumptions (A2), (A3) and (A4)
\begin{align*}
c(\tConf)
 \leq
2  \sum_{j, \Delta_j \not \subseteq \Lambda_N}
N_{\Delta_j}(\tConf) 
\sum_{j', \Delta_{j'} \subseteq \Lambda}
\psi_{|j-j'|}
 \int 
 N_{\Delta_{j'}}(\tConf'_{\Lambda})
\Specification^{\Intensity, \Proportion}_{\Lambda, \tConf}(d \tConf').
\end{align*}
\begin{remark}
In the last bound, we used from (A4) that the potential $\psi(x)$ is non-positive 
when $|x|$ is large.
One can do without this assumption and would get an extra factor 2.
\end{remark}
In the following, when not specified, $j,k$ are indexes such that 
$\Delta_j, \Delta_k \not \subseteq \Lambda$ 
and $j',k'$ are such that $\Delta_{j'},\Delta_{k'} \subseteq \Lambda$.

Now let 
$B_{\tConf}:= 
\sup \{
b + \sum_{j, \Delta_{j} \not \subseteq \Lambda} 
\psi_{|j-k'|} N_{\Delta_j}(\tConf) \ | \
\Delta_{k'}\subseteq \Lambda \}$
with $a,b$ coming from assumption (A3).
Then we have
\begin{align*}
 \int 
 N_{\Delta_{j'}}(\tConf'_{\Lambda})
\Specification^{\Intensity, \Proportion}_{\Lambda, \tConf}(d \tConf')
\leq
\frac{2 B_{\tConf}}{a} +
\underbrace{\int 
 N_{\Delta_{j'}}(\tConf'_{\Lambda}) 
 \1_{ \{
 N_{\Delta_{k'}}(\tConf')>\frac{2B_{\tConf}}{a}, \forall k '
 \}}
\Specification^{\Intensity, \Proportion}_{\Lambda, \tConf}(d \tConf')}_{*}.
\end{align*}
But from assumption (A1) and (A3)
\begin{align*}
* \leq &
e^{\Intensity |\Lambda|}
\int 
N_{\Delta_{j'}}(\tConf'_{\Lambda}) \1_{ \{
N_{\Delta_{k'}}(\tConf')>\frac{2B_{\tConf}}{a}, \forall k'
\}} 
\\ & \hspace{2cm}
\exp\left( \sum_{k'} - a N_{\Delta_{k'}}(\tConf')^2
+B_{\tConf}
N_{\Delta_{k'}}(\tConf')
\right)
\Poisson_{\Lambda}^{\Intensity, \Proportion}(d \tConf)
\\ & \leq 
e^{\Intensity |\Lambda|}
\int 
N_{\Delta_{j'}}(\tConf'_{\Lambda})
\exp\left( -\frac{a}{2}\sum_{k'} N_{\Delta_{k'}}(\tConf')^2
\right)
\Poisson_{\Lambda}^{\Intensity, \Proportion}(d \tConf)
\\ &   \leq \Constant < \infty
\end{align*}
and we finally obtain
\begin{align*}
c(\tConf)
\leq 
\left(
\Constant ' + \Constant ''
\sup \{\sum_{k} 
\psi_{|k-k'|} N_{\Delta_k}(\tConf) \ | \
\Delta_{k'}\subseteq \Lambda \}
\right)
\sum_{j, \Delta_j \not \subseteq \Lambda_N}
\sum_{j'} \psi_{|j-j'|}
N_{\Delta_j}(\tConf) ,
\end{align*}
and from the Cauchy-Schwarz inequality and the Lemma \ref{lemma_second_moment_candidate} we obtain the wanted result.

\vspace{1cm}

{\it Acknowledgement:} This work was supported in part by the  ANR project PPPP (ANR-16-CE40-0016) and by Deutsche Forschungsgemeinschaft (DFG) through grant CRC 1294 "Data Assimilation", Project A05.

\bibliographystyle{plain}
\bibliography{biblio}
\end{document}

%% file: macros.tex
\newcommand{\ensemblenombre}[1]{\mathbb{#1}}
\newcommand{\N}{\ensemblenombre{N}} 
\newcommand{\Z}{\ensemblenombre{Z}} 
\newcommand{\R}{\ensemblenombre{R}}
\newcommand{\Dim}{d}
\newcommand{\Color}{q}

\newcommand{\Conf}{\omega}
\newcommand{\tConf}{\widetilde{\Conf}}

\newcommand{\ConfSpace}{\Omega}
\newcommand{\TemperedConfSpace}{\mathcal{T}}

\newcommand{\tConfSpace}{\widetilde{\ConfSpace}}

\newcommand{\tProbaSetInvariant}{\widetilde{\mathcal{P}}_\theta}
\newcommand{\tProbaSetInvariantDelta}{\widetilde{\mathcal{P}}_{\theta_{\delta}}}
\newcommand{\A}{\mathcal{A}} 
\newcommand{\1}{\mathds{1}}
\newcommand{\GibbsPotts}{\mathcal{G}^{potts}}
\newcommand{\EdgeSet}{\mathcal{E}}
\newcommand{\Poisson}{\pi}
\newcommand{\tPoisson}{\widetilde{\Poisson}}

\newcommand{\Intensity}{z}

\newcommand{\Proportion}{\tilde{\alpha}}

\newcommand{\ProportionMax}{\#^{\Proportion}_{\max}}
\newcommand{\RCM}{\mathcal{C}^{\Intensity, \Proportion}_{\Lambda ,wired}}
\newcommand{\RCMn}{\mathcal{C}^{\Intensity, \Proportion}_{\Lambda_n ,wired}}
\newcommand{\FK}{\mathbb{P}^{\Intensity,\Proportion}_{\Lambda}}
\newcommand{\FKA}{\mathbb{P}^{\Intensity,\Proportion}_{\Lambda,\A}}

\newcommand{\Constant}{\kappa}
\newcommand{\ConstantBis}{\tilde{\Constant}}
\newcommand{\InfiniteCluster}{C_{\infty}}

\newcommand{\Specification}{\mathscr{P}}
\newcommand{\tSpecification}{{\Xi}}
\newcommand{\PartitionFunction}{Z^{\Intensity,\Proportion}}
\newcommand{\Hamiltonian}{H}

\newcommand{\Entropy}{\mathcal{I}}
